\documentclass[reqno,10pt]{amsart}
\usepackage{amscd,amssymb,amsmath,color,url}
\usepackage{latexsym}
\usepackage[all]{xy}

\def\into{\hookrightarrow}

\def\rra{\rightrightarrows}
\def\Ra{\Rightarrow}
\def\la{\leftarrow}

\def\Lra{\Leftrightarrow}
\def\toisom{\widetilde{\to}}

\def\.{,\dotsc ,}
\def\:{\colon}
\def\wt{\widetilde}
\def\wh{\widehat}
\def\ol{\overline}

\def\Mor{{\rm Mor}}

\def\Spec{{\rm Spec}}

\def\Frac{{\rm Frac}}
\def\bfSpec{{\bf Spec}}
\def\bfProj{{\bf Proj}}
\def\Bl{{\rm Bl}}

\def\an{{\rm an}}

\def\Ker{{\rm Ker}}

\def\Im{{\rm Im}}

\def\id{{\rm id}}

\def\bfA{{\bf A}}

\def\bfF{{\bf F}}

\def\bfN{{\bf N}}

\def\bfP{{\bf P}}
\def\bfQ{{\bf Q}}

\def\bfZ{{\bf Z}}

\def\gtC{{\mathfrak C}}

\def\calC{{\mathcal C}}

\def\calI{{\mathcal I}}
\def\calJ{{\mathcal J}}
\def\calK{{\mathcal K}}
\def\calL{{\mathcal L}}

\def\calO{{\mathcal O}}

\def\calU{{\mathcal U}}
\def\calV{{\mathcal V}}

\def\calX{{\mathcal X}}

\def\oU{{\ol U}}

\def\oY{{\ol Y}}
\def\oZ{{\ol Z}}

\def\ou{{\ol u}}

\def\ox{{\ol x}}

\def\tilT{{\wt T}}

\def\tilX{{\wt X}}

\def\tilx{{\wt x}}

\def\hatA{{\wh A}}

\def\hatI{{\wh I}}

\def\tileta{{\wt\eta}}

\def\alp{{\alpha}}

\def\veps{\varepsilon}

\def\R+*{{\bf R^*_+}}

\def\colim{{\rm{colim}}}

\newtheorem{theor}[subsubsection]{Theorem}
\newtheorem{prop}[subsubsection]{Proposition}
\newtheorem{lem}[subsubsection]{Lemma}
\newtheorem{cor}[subsubsection]{Corollary}
\newtheorem{claim}[subsubsection]{Claim}

\theoremstyle{definition}

\newtheorem{rem}[subsubsection]{Remark}

\newtheorem{exam}[subsubsection]{Example}

\def\colim{{\rm colim}}

\def\FP{{\rm FP}}

\begin{document}

\author{Michael Temkin, Ilya Tyomkin}
\thanks{Both authors were supported by the Israel Science Foundation (grant No. 1018/11).}
\title{Pr\"ufer algebraic spaces}

\address{Einstein Institute of Mathematics, The Hebrew University of Jerusalem, Giv'at Ram, Jerusalem, 91904, Israel}
\email{temkin@math.huji.ac.il}
\address{Department of Mathematics, Ben-Gurion University of the Negev, P.O.Box 653, Be'er Sheva, 84105, Israel}
\email{tyomkin@math.bgu.ac.il}
\keywords{Pr\"ufer pairs, algebraic spaces.}
\begin{abstract}
This is the first in a series of two papers concerned with relative birational geometry of algebraic spaces. In this paper, we study Pr\"ufer spaces and Pr\"ufer pairs of algebraic spaces that generalize spectra of Pr\"ufer rings. As a particular case of Pr\"ufer spaces we introduce valuation algebraic spaces, and use them to establish valuative criteria of separatedness and properness that sharpen the standard criteria. In a sequel paper, we introduce a version of Riemann-Zariski spaces, and prove Nagata's compactification theorem for algebraic spaces.
\end{abstract}

\maketitle

\section{Introduction}

This is the first in a series of two papers devoted to the study of relative birational geometry, Riemann-Zariski spaces ({\em RZ spaces}), and Nagata's compactification theorem for algebraic spaces. We generalize the ideas and methods of \cite{temrz} to the category of algebraic spaces, and hope that similar techniques will apply to representable morphisms of stacks, and may also be useful for studying non-representable morphisms. Our current aim is to sharpen the methods in the relatively simple context of algebraic spaces.

Valuation rings and their spectra, that we call {\em valuation schemes}, play an important role in the theory of RZ spaces in general, and in \cite{temrz} in particular. One reason for this is that valuation schemes are atomic objects for the topology of modifications of integral schemes, namely valuation schemes are precisely the integral local schemes that do not possess non-trivial modifications. In the category of algebraic spaces one often passes to \'etale covers, hence it is more natural to work with semi-local analogs of valuation schemes, or more generally with {\em Pr\"ufer spaces}.

In the current paper we introduce Pr\"ufer algebraic spaces and pairs, and study their basic properties. As an application of the theory we prove refined valuative criteria of separatedness and properness for algebraic spaces. The results of this paper play an important role in our sequel paper, in which we introduce a version of Riemann-Zariski spaces for morphisms of algebraic spaces and give a new proof of Nagata's compactification theorem for algebraic spaces. Since the class of Pr\"ufer spaces seems to be interesting on its own, we devote the entire paper to it, and go far beyond what we need in \cite{nag}.

Recall that an integral domain is called Pr\"ufer if and only if any of its localizations with respect to a prime ideal is a valuation ring. Many equivalent algebraic descriptions of Pr\"ufer domains can be found in the literature, e.g., $14$ such conditions are listed in \cite[Ch. VII, \S2, Exercise 12]{Bou}. Geometrically, Pr\"ufer domains are precisely the domains whose spectra admit no non-trivial modifications. This point of view allows one to define the notion of Pr\"uferness in the context of algebraic spaces, and even stacks. Furthermore, since in birational geometry one often studies modifications preserving a certain subset, it is natural to introduce the notion of Pr\"ufer pair to be a pair $U\subset X$ satisfying certain technical assumptions and such that $X$ admits no non-trivial $U$-modifications. In the affine case, this geometric definition is equivalent to the notion of $R$-Pr\"ufer ring (cf. \cite{KZ}).

Our study of Pr\"ufer pairs involves a generalization of the class of open immersions, which we call {\em pro-open immersions}. These are morphisms $i\:U\to X$ such that any morphism $f\:Y\to X$ with $f(Y)\subseteq i(U)$ factors uniquely through $i$. The particular case of quasi-compact pro-open immersions of schemes was studied by Raynaud, and the general case is developed in the first part of the current paper.

The second part of the paper is devoted to the study of basic properties of Pr\"ufer spaces and pairs. The main results here are Theorem~\ref{prufpairth} and Corollary~\ref{prufspaceth} giving several equivalent characterizations of Pr\"ufer pairs and spaces, and Proposition~\ref{sepprufprop} asserting that separated Pr\"ufer algebraic spaces are necessarily schemes. Moreover, separated semi-local Pr\"ufer algebraic spaces are affine.

Finally, in the last part of the paper we define valuative spaces to be Pr\"ufer spaces with a unique closed point, and use the developed theory to prove the refined versions of valuative criteria of separatedness and properness for algebraic spaces using only Zariski valuative diagrams, i.e., diagrams in which the generic point of the valuative space is a Zariski point of the algebraic space (Propositions~\ref{sepcritprop} and \ref{valcritprop} and Theorem~\ref{valcritth}).

\subsection{The structure of the paper}
In Section~\ref{prelimsec}, we recall basic facts about algebraic spaces, and approximation theory.

In Section~\ref{prosec}, we define the {\em pro-open immersions} via the functors of points, and characterize them in Theorems~\ref{topth} and \ref{critproopen} up to a minor restriction as those flat monomorphisms that induce topological embeddings of the underlying topological spaces. In addition, we describe in \S\ref{ftsec} pro-open immersions of finite type. Finally, we establish approximation theory for pro-open immersions, but since it is not used in the paper, this is done in the appendix.

In Section~\ref{pairsec}, we introduce {\em Pr\"ufer algebraic spaces and pairs}. Theorem~\ref{prufpairth} and Corollary~\ref{prufspaceth} provide $8$ equivalent descriptions of Pr\"ufer pairs and spaces. In particular, we show that for an integral qcqs algebraic space $X$, the following are equivalent: (a) $X$ is Pr\"ufer, (b) $X$ does not admit non-trivial blow ups, (c) some, and hence any, affine presentation of $X$ is the spectrum of a Pr\"ufer ring.

Finally, we study valuation algebraic spaces in Section~\ref{valsec}. We show that if a valuation space $X$ is separated then it is the spectrum of a valuation ring, and give various examples of non-separated (hence non-schematic) valuation spaces. In \S\ref{valcritsec}, we establish a Zariski-local valuative criterion of universal closedness, in which one does not extend the fraction field of the valuation. Note that in order to achieve Zariski-locality it is necessary to deal with all valuation algebraic spaces, including the non-separated ones. Finally, we prove a stronger criterion of separatedness (Propositions~\ref{sepcritprop}) and properness (Theorem~\ref{valcritth}). The latter is an analog of \cite[Proposition 3.2.3]{temrz}, and this is the criterion we use in \cite{nag}.

\subsection*{Acknowledgements}
We would like to thank the anonymous referee for pointing out some gaps in the first version of the paper, proposing a significant simplification of the proof of Proposition~\ref{sepcritprop}, and suggesting an example of a pro-open immersion, which is not a topological embedding, see \S\ref{wildproopensec}.

\section{Preliminaries}\label{prelimsec}

\subsection{Algebraic spaces}
{\em We work only with quasi-separated algebraic spaces.} Basic references to algebraic spaces are \cite{Knu} and \cite{stacks}. The underlying topological space of an algebraic space $X$ will often be denoted by $|X|$. We say that $X$ is {\em local} if $|X|$ is quasi-compact and has a unique closed point. 

\subsubsection{Presentations}
If not said to the contrary, a presentation of an algebraic space $X$ means an \'etale presentation, i.e., a surjective \'etale morphism $X_0\to X$ whose source is a scheme. If $X_0$ is affine then we say that the presentation is {\em affine}. Giving a presentation is equivalent to giving an \'etale equivalence relation $X_1\rra X_0$ with an isomorphism $X_0/X_1\toisom X$, so the latter will often be used to refer to the presentation. If $X$ is qcqs then we automatically consider only its quasi-compact presentations, so the word quasi-compact will usually be omitted.

\subsubsection{Zariski points}\label{zarpointssec}
An algebraic space $\eta$ is a {\em point} if any monomorphism to $\eta$ is an isomorphism. It is well known that any point is the spectrum of a field. A {\em point} of an algebraic space $X$ is a morphism $f\:\eta\to X$ from a point. If $f\:\Spec(K)\to X$ is a monomorphism then we say that $f$ is a {\em Zariski point} and call $k(\eta):=K$ the residue field of $\eta$. The underlying topological space $|X|$ of an algebraic space $X$ is the set of isomorphism classes of Zariski points provided with a suitable topology, see \cite[Ch. 2, \S6]{Knu}. We will often use the following two facts \cite[Ch.~2, Proposition 6.2 and Theorem 6.4]{Knu}: (1) any point $f\:\eta\to X$ factors uniquely through a Zariski point $\eta_0\to X$, and (2) any point of an algebraic space $X$ factors through some \'etale presentation $X'\to X$.

\subsubsection{Schematic image and dominance}\label{subsec:schemdom}
Let $f\:Y\to X$ be a morphism of algebraic spaces. Recall that the {\em scheme-theoretic image}, or schematic image, $\Im(f)$ of the morphism $f$ is the minimal closed subspace $X'\into X$ through which $f$ factors. It always exists by \cite[Tag:082X]{stacks}. Note that even for morphisms of schemes, this fact in such generality is missing in \cite[\S6.10]{egaI}. If $\Im(f)=X$ then we say that $f$ is {\em schematically dominant}. The following lemma about the transitivity of the schematic image generalizes \cite[Proposition 6.10.3]{egaI} to the case of algebraic spaces. The proof is identical to \cite[Proposition 6.10.3]{egaI}, hence we omit it.

\begin{lem}\label{translem}
Let $g\:Z\to Y$ and $f\:Y\to X$ be morphisms of algebraic spaces. Set $Y':=\Im(g)$, and let $f'\:Y'\to X$ be the restriction of $f$ to $Y'$. Then $\Im(f')=\Im(f\circ g)$. In particular, if $g$ is schematically dominant then $\Im(f)=\Im(f\circ g)$, and hence $f$ is schematically dominant if and only if so is $f\circ g$.
\end{lem}

In general, as explained in \cite[Tag:01QW and Tag:01R6]{stacks}, schematic images may behave wildly. However, if $f$ is quasi-compact, as will always be the case in this paper, then the ideal $\calI_f=\Ker(\calO_X\to f_*\calO_Y)$ is quasi-coherent, and $\Im(f)$ is just the closed subspace corresponding to $\calI_f$, moreover, taking the schematic image commutes with \'etale base changes, see \cite[Tag:082Z]{stacks}.

\begin{lem}\label{lem:extofid}
Let $X$ be a scheme, $x\in X$ a point, $I_x\subset \calO_{X,x}$ an ideal, and $Z_x$ the scheme defined by $I_x$. Let $f\:Z_x\to X$ be the natural morphism, $Z$ its schematic image, and $\calJ\subset \calO_X$ the ideal of $Z$. Then (1) $I_x$ is the stalk of $\calJ$ at $x$, (2) the support of $Z$ is the closure of $|Z_x|\subset |X|$, and (3) if $I_x$ is of finite type then there exists a direct system of ideals of finite type $\calI_\alp$ such that $\colim \calI_\alp=\calJ$ and the stalk of $\calI_\alp$ at $x$ is $I_x$ for any $\alp$.
\end{lem}
\begin{proof}
(1) Since taking the schematic image commutes with \'etale base changes, we may assume that $X=\Spec A$. Since $f$ is quasi-compact, $\calJ$ is the sheaf associated to the kernel of $A\to \calO_{X,x}/I_x$, i.e., to the preimage of $I_x$ in $A$ under the map $A\to \calO_{X,x}$. Clearly, its stalk at $x$ is $I_x$.

(2) It suffices to show that the support of $Z$ is contained in the closure of $|Z_x|$ in $|X|$. Let $U\subset X$ be the complement of the closure of $|Z_x|$. Then $U\times_XZ$ is the schematic image of $U\times_XZ_x=\emptyset$. Hence $U\cap Z=\emptyset$.

(3) The ideal $\calJ$ is the direct limit of subideals $\calJ_\alp$ of finite type. If we find a subideal $\calI\subseteq \calJ$ of finite type such that $\calI_x=I_x$ then the direct system $\calI_\alp:=\calJ_\alp+\calI$ satisfies the assertion. To construct $\calI$, we extend $I_x$ to an ideal of finite type on an open affine neighborhood of $x$, and by \cite[Theorem~6.9.7]{egaI}, this extension extends to an ideal $\calI\subset \calO_X$ of finite type. Notice that $\calI\subseteq\calJ$ since $\calJ$ is the kernel of $\calO_X\to f_*\left(\calO_{X,x}/I_x\right)$.
\end{proof}

\begin{lem}\label{flatbaselem}
If $f\:Y\to X$ is quasi-compact and schematically dominant then so is any of its flat base changes.
\end{lem}
\begin{proof}
If $f'\:Y'\to X'$ is the base change with respect to a flat morphism $g\:X'\to X$ then $f'_*\calO_{Y'}=g^*f_*\calO_Y$, and hence $\calI_{f'}=\calI_f\calO_{X'}$.
\end{proof}

Let $X$ be a quasi-compact algebraic space and $U\into X$ an open subspace. We say that $U$ is {\em schematically dense} in $X$ if for any \'etale morphism $f\:V\to X$ the morphism $f^{-1}(U)\to V$ is schematically dominant \cite[Tag:0834]{stacks}. Note that by Lemma~\ref{flatbaselem}, if $U\to X$ is quasi-compact then $U$ is schematically dense in $X$ if and only if $U\to X$ is schematically dominant.

\subsubsection{Limits and pushouts}
We say that a pushout (resp. a limit, resp. a colimit, etc.) {\em exists} in a category $\gtC$ if it is representable by an object $X$ of $\gtC$. In such case, we will freely say that $X$ is the pushout (resp. the limit, the colimit, etc.) although, strictly speaking, it is defined only up to a unique isomorphism.

\subsubsection{Approximation}\label{approxsec}
All algebraic spaces in this section are assumed to be {\em qcqs}. Throughout the paper we will use various results of  what we call {\em approximation theory}, that studies filtered limits of geometric objects with affine transition morphisms. The most complete and updated form of this theory can be found in \cite{rydapr}, as well as the history of the subject and references to other sources. Let us recall briefly the main results for algebraic spaces. {\em Classical approximation} studies filtered projective families of spaces $\{X_\alp\}$ with affine transition morphisms. It was developed for schemes in \cite[$\rm IV_3$, \S8]{ega}, and the case of general spaces (and even stacks) follows rather easily, see \cite[Appendix B]{rydapr}. In particular: (a) there exists a limit $X=\lim_\alp X_\alp$, $|X|\toisom\lim_\alp|X_\alp|$, and the projections $X\to X_\alp$ are affine, (b) the category $\FP/X$ of finitely presented $X$-spaces is naturally equivalent to the 2-categorical colimit of the categories $\FP/X_\alp$, and for almost all geometric properties (e.g., projectivity, smoothness, etc.) a morphism $f\:Y\to Z$ in $\FP/X$ satisfies $\bfP$ if and only if for sufficiently large $\alp$ its approximations $f_\alp\:Y_\alp\to Z_\alp$ in $\FP/X_\alp$ satisfy $\bfP$, (c) if $\{X_\alp\}$ is an $S$-family and $Y\to S$ is finitely presented then $\colim_\alp\Mor_S(X_\alp,Y)\toisom\Mor_S(X,Y)$.

More subtle approximation results are related to filtered families that are subject to certain finite presentation or finite type restrictions. {\em Approximation of modules}, see \cite[Theorem A]{rydapr}: any finitely generated module is an epimorphic image of a finitely presented one, and any quasi-coherent $\calO_X$-module $\calL$ is the filtered colimit of both: the family of finitely generated submodules, and the family of finitely presented $\calO_X$-modules with a morphism to $\calL$. The same claims hold for quasi-coherent $\calO_X$-algebras. {\em Approximation of morphisms}, see \cite[Theorem D]{rydapr}: any finite type morphism $Y\to X$ factors as a closed immersion $Y\to\oY$ followed by a finitely presented morphism $\oY\to X$. Any morphism of algebraic spaces $f\:X\to S$  is a filtered limit of finitely presented morphisms $f\:X_\alp\to S$ with affine transition morphisms $X_\alp\to X_\beta$, and for a list of properties $\bfP_0$ stable under composition and including affine morphisms (e.g., affine, quasi-affine or separated) $f$ satisfies $\bfP_0$ if and only if so do $f_\alp$ for all $\alp$ large enough. {\em Approximation of properties}, see \cite[Theorem C]{rydapr}: if all $X_\alp$ are finitely presented over $S$ then $X\to S$ satisfies $\bfP_0$ if and only if so do $X_\alp\to S$ for all $\alp$ large enough.

\subsubsection{Schematically dominant morphisms and approximation}\label{schdomsec}
Some approximation results impose a finite presentation assumption that cannot be weakened to a finite type assumption. However, for schematically dominant morphisms we have the following lemmas:

\begin{lem}\label{thD}
Let $Y\to X$ be a morphism between qcqs algebraic spaces. Then there exists a filtered family $\{Y_\alp\}$ of qcqs algebraic spaces of finite type over $X$ such that the transition morphisms are affine and schematically dominant and there is an isomorphism of $X$-spaces $Y\toisom\underleftarrow{\lim}_{\alp}Y_\alp$. In particular, the projections $Y\to Y_\alp$ are affine and schematically dominant. In addition, if $Y$ is $X$-separated one can choose $Y_\alp$ to be $X$-separated.
\end{lem}
\begin{proof}
By approximation (see \cite[Theorem D(i)]{rydapr}), $Y$ is $X$-isomorphic to the filtered limit of finitely presented $X$-spaces $Y'_\alp$ with affine transition morphisms, and $Y'_\alp$ can be chosen $X$-separated whenever $Y$ is $X$-separated. It remains to set $Y_\alp$ to be the schematic image of the projection $Y\to Y'_\alp$. Since $Y'_\alp$ are finitely presented $X$-spaces they are qcqs, and hence so are $Y_\alp$.
\end{proof}

\begin{lem}\label{propB}
Assume that $X$ is a qcqs algebraic space, $\{Y_\alp\}$ a filtered family of qcqs algebraic spaces over $X$ with schematically dominant affine transition morphisms and limit $Y$, and $Z$ an $X$-space of finite type. Then the natural map $\colim_\alp\Mor_X(Y_\alp,Z)\to\Mor_X(Y,Z)$ is an isomorphism.
\end{lem}
\begin{proof}
By \cite[Theorem D(a) and Appendix B]{rydapr}, the space $Z$ can be embedded as a closed subspace into an algebraic $X$-space $Z'$ of finite presentation, and $\colim_\alp\Mor_X(Y_\alp,Z')\toisom\Mor_X(Y,Z')$. We clearly have $\colim_\alp\Mor_X(Y_\alp,Z)\subseteq\colim_\alp\Mor_X(Y_\alp,Z')$ and $\Mor_X(Y,Z)\subseteq\Mor_X(Y,Z')$. Finally, if $h\:Y_\alp\to Z'$ is such that the composition $Y\to Y_\alp\to Z'$ factors through $Z$ then $h$ factor through $Z$ since the projection $Y\to Y_\alp$ is schematically dominant.
\end{proof}

\subsection{Spectral spaces}

\subsubsection{Specialization relation}
Given a topological space $X$ we denote by $\succeq$ and $\preceq$ the generalization and the specialization relations on $X$, and by $\succ$ and $\prec$ the corresponding strict relations. The following notation will be useful: $X_{\succeq y}:=\{x\in X\, |\, x\succeq y\}$, $X_{\succ y}:=\{x\in X\, |\, x\succ y\}$, $X_{\preceq y}:=\{x\in X\, |\, x\preceq y\}$, $X_{\prec y}:=\{x\in X\, |\, x\prec y\}$. In particular, $X_{\preceq y}$ is the closure of $y$, and if $X$ is a scheme then $X_{\succeq y}$ is the underlying topological set of the local scheme $X_{y}=\Spec(\calO_{X,y})$.

\subsubsection{$S$-topology and generizing sets}\label{genersec}
The specialization relation induces a topology, which we, following Rydh, call $S$-topology, see \cite[Section~1]{rydhsubmersions}. Its open (resp. closed) sets are the sets closed under generalization (resp. specialization). For brevity, we call $S$-open sets {\em generizing}. Obviously, the family of such sets is closed under arbitrary unions and intersections, and a set is generizing if and only if it is an intersection of Zariski open sets. For any morphism of algebraic spaces $f\:Y\to X$ the underlying continuous map $|f|\:|Y|\to|X|$ preserves the specialization relation, hence is continuous in the $S$-topology. Thus, the preimage of a generizing set is generizing. In the opposite direction, any open or flat morphism $f$ is $S$-open (or {\em generizing}) by \cite[Corollary~5.7.1 and Proposition~5.8]{LMB}. In particular, if $f$ is flat then $|f|(|Y|_{\succeq y})=|X|_{\succeq |f|(y)}$ for any $y\in |Y|$, and the image of a generizing set is generizing.

\subsubsection{Spectral spaces}
In his thesis, Hochster gave a topological characterization of the spectra of rings and called such spaces spectral spaces (see \cite{Hochster}). A topological space is called {\em spectral} if it is quasi-compact, sober, the intersection of a pair of quasi-compact opens is quasi-compact, and the collection of quasi-compact opens forms a basis for the topology. A continuous map of spectral spaces is called {\em spectral} if the preimage of every open quasi-compact subset is quasi-compact. Recall that if $X$ is a qcqs algebraic space then the underlying topological space $|X|$ is spectral by \cite[Tag:0A4G]{stacks}.

\subsubsection{Zariski trees}
Let $T$ be a spectral space. We say that $T$ is a {\em Zariski tree} if $T$ is connected and for any point $t\in T$ the set $T_{\succeq t}$ is totally ordered by specialization. In particular, if $X$ is a qcqs algebraic space and $|X|$ is a Zariski tree then $X$ is irreducible. To justify our terminology, we can view $T$ as a tree directed by the specialization. Its root is the generic point and its leaves are the closed points. A \emph{Zariski forest} is a disjoint union of finitely many Zariski trees. A basic example of a Zariski tree is an irreducible algebraic curve or (as we will see later) an irreducible Pr\"ufer scheme.

\subsubsection{Zariski chains} \label{chainsec}
By a {\em Zariski chain} we mean a Zariski tree $T$ such that the set of specializations of any point is totally ordered. Since we assume that $T$ is quasi-compact this simply means that $T$ has a unique closed point. One easily sees that any generizing subset $S\subseteq T$ is of the form $T_{\succeq t}$ or $T_{\succ t}$, and $S$ is open (resp. quasi-compact) if and only if it is of the form $T_{\succ t}$ or coincides with $T$ (resp. is of the form $T_{\succeq t}$ or $S=\emptyset$). In particular, $T_{\succeq t}$ is open if and only if $t$ is closed or possesses an immediate specialization. A typical example of a Zariski chain is $\Spec(R)$ for a valuation ring $R$, because all ideals of $R$ are totally ordered by inclusion.

\begin{lem}\label{lem:etcovofvalsch}
Let $\phi\:Y\to X$ be an essentially \'etale morphism, and assume that $X=\Spec(R)$ for a valuation ring $R$ and $Y=\Spec(A)$
for a local ring $A$. Then $A$ is a valuation ring, and the induced map $|\phi|\:|Y|\to |X|$ is a topological embedding with generizing image. In particular, $|\phi|$ is a homeomorphism if and only if its image contains the closed point of $X$.
\end{lem}
\begin{proof}
Note that $Y$ is normal because it is essentially \'etale over the normal scheme $X$. It follows from the locality of $A$ that $Y$ is irreducible. By Zariski's main theorem, $A$ is a localization of a finite $R$-algebra, and using that $A$ is integrally closed in $k(Y)$ we obtain that $A$ is a localization of the integral closure of $R$ in $k(Y)$. Thus, $A$ is a valuation ring by  \cite[Ch.VI, \S7, Propositions 1 and 6]{Bou} and $Y$ is a Zariski chain. Since $|\phi|$ is a generizing map (see \S\ref{genersec}), it remains to show that it is injective. Note that $\phi$ has discrete fibers, since it is a localization of a finite morphism. However, the only discrete subsets in a Zariski chain are points.
\end{proof}

\subsubsection{The constructible topology}\label{constrsec}
Given a spectral space $X$, one defines the {\em constructible (or patch) topology} on $X$ to be the topology generated by the open quasi-compact subsets of $X$ and their complements. In particular, open sets are ind-constructible and closed sets are pro-constructible. By \cite[Theorem 1]{Hochster}, any spectral space is compact with respect to the constructible topology.

\subsubsection{Some facts}\label{factsec}
We will need the following facts about a spectral space $X$ and its topologies:

(i) {\em A subset $T\subseteq X$ is generizing and quasi-compact if and only if $T$ is the intersection of a collection of open quasi-compact subsets.} The check is straightforward and (when $X$ is a scheme) the claim is \cite[Lemma 2.1]{Ray}.

(ii) {\em If $T$ is as in (i) and $f\:Y\to X$ is a spectral map then $f^{-1}(S)$ is generizing and quasi-compact.} This follows from (i).

(iii) {\em The Zariski topology of $X$ is the intersection of the $S$-topology and the constructible topology in the sense that it is the strongest topology which is weaker than these two.} See \cite[Corollary~1.5]{rydhsubmersions}.

(iv) {\em If $X$ is a qcqs algebraic space then a subset $Z\subseteq|X|$ is closed in the constructible topology if and only if there exists a morphism $f\:Y\to X$ such that $Y$ is an affine scheme and $Z=f(Y)$.} The claim reduces easily to the case when $X$ is a scheme (even an affine one), which is established in \cite[Proposition~7.2.1]{egaI}.

\subsection{Modifications and quasi-modifications}

\subsubsection{$U$-admissibility}
Let $f\:X'\to X$ be a morphism of qcqs spaces, and $U\subseteq|X|$ a quasi-compact generizing subset. Then $f$ is quasi-compact, and $f^{-1}(U)$ is quasi-compact and generizing by \S\ref{factsec}(ii). We say that $U$ is {\em schematically dense} in $X$ if any quasi-compact open subspace containing $U$ is so (cf. \S\ref{subsec:schemdom}). We say that $f$ is {\em $U$-admissible} if $f^{-1}(U)$ is schematically dense in $|X'|$.

\subsubsection{$U$-modifications}\label{qmodsec}
Let $X$ be a qcqs space, and $U\subseteq|X|$ a schematically dense quasi-compact generizing subset. By a {\em $U$-modification} (resp. {\em $U$-quasi-modification}) we mean a proper (resp. separated finite type) $U$-admissible morphism $f\:X'\to X$ with qcqs source such that there exists an open subspace $V\subseteq X$ containing $U$ for which  $f^{-1}(V)\to V$ is an isomorphism (resp. an open immersion). By a {\em strict} $U$-quasi-modification $X'\to X$ we mean a $U$-quasi-modification whose restriction onto a quasi-compact neighborhood of $U$ in $X$ is an isomorphism. In general, by a {\em modification} (resp. {\em quasi-modification}) we mean a morphism $f\:X'\to X$ which is a $U$-modification (resp. $U$-quasi-modification) for some choice of a schematically dense quasi-compact generizing subset $U\subseteq|X|$.

\begin{rem}
If $X'$ and $X''$ are two $U$-modifications of $X$ then there exists at most one $X$-morphism $f\:X''\to X'$ because $U$ is schematically dense in $X''$ and $X'$ is $X$-separated. If such $f$ exists we say that $X''$ {\em dominates} $X'$. The set of all $U$-modifications of $X$ ordered by domination is filtered because $U$-modifications $X',X''$ are dominated by the schematic image of $U$ in $X'\times_XX''$. The same facts hold true for the family of modifications of $X$.
\end{rem}

\subsubsection{Blow ups}
The {\em blow up} of $X$ along an ideal $\calI\subset\calO_X$ is defined as $\Bl_\calI(X):=\bfProj(\oplus_{n=0}^\infty\calI^n)$. If $X$ is reduced and the closed immersion $Z=\bfSpec(\calO_X/\calI)\into X$ is finitely presented and nowhere dense then the blow up $f\:\Bl_\calI(X)\to X$ is a modification. If, in addition, $\calI$ is {\em $U$-trivial} in the sense that $|Z|$ is disjoint from $U$ then $f$ is a $U$-modification.

\section{Pro-open immersions}\label{prosec}
In this section we introduce and study pro-open immersions. In the quasi-compact case, the families of pro-open immersions and flat monomorphisms coincide, but the former family behaves better in general. In the category of schemes, quasi-compact flat monomorphisms were studied by Raynaud in \cite{Ray}. The general case can be reduced to Raynaud's results once one proves that pro-open immersions are schematic.

\subsection{Basic facts}\label{defprosec}

\subsubsection{The definition}
Let $X$ be an algebraic space, and $U\subseteq |X|$ be a generizing subset. Consider the functor $h_U$ that associates to an algebraic space $Y$ the set of morphisms $Y\to X$ having set-theoretic image in $U$. If $h_U$ is representable by an algebraic space $\calU$ then we say that $U$ is a {\em pro-open subspace}. Any morphism isomorphic to $\calU\to X$ as above is called a {\em pro-open immersion}.

\subsubsection{Localization of algebraic spaces}
A typical example of a pro-open immersion is the localization morphism $\Spec(\calO_{X,x})\to X$ for a scheme $X$ and a point $x\in X$. More generally, let $X$ be an algebraic space with a point $x$. If $U=X_{\succeq x}$ is a pro-open subspace then we call the corresponding algebraic space $\calU$ the {\em localization} of $X$ at $x$. One of the technical obstacles in the proof of Nagata's compactification theorem in \cite{nag} is that localization does not exist in general. A tightly related fact is that the semilocalization of a scheme $X$ at points $x_1\..x_n$ does not have to exist: the generizing set $\cup_{i=1}^nX_{\succeq x_i}$ does not have to be a pro-open subscheme. A counter-example can be obtained already for Hironaka's smooth proper threefold, see \cite[Appendix B, Example~3.4.2]{Hartshorne}. This was discovered in our discussion with D. Rydh, and a detailed presentation will be given elsewhere.

\subsubsection{First properties}
Here is a list of simple properties of pro-open immersions:

\begin{prop}\label{proprop}
(i) If $i\:\calU\to X$ is a pro-open immersion, and $U\subseteq |X|$ is its image then $|\calU|\to U$ is a bijection.

(ii) A generizing subset $U\subseteq|X|$ is a pro-open subspace if and only if the filtered intersection of the open subspaces $\calU_\alp\into X$ for which $U\subseteq|\calU_\alp|$ is representable. In this case, $\cap_\alp \calU_\alp$ is the pro-open subspace corresponding to $U$.

(iii) Any pro-open immersion $i\:\calU\to X$ is a flat monomorphism.

(iv) The class of pro-open immersions is closed under compositions and base changes.

(v) If $i\:\calU\to X$ and $j\:\calV\to X$ are pro-open immersions and $i(|\calU|)\subseteq j(|\calV|)$ then the natural morphism $\calU\to \calV$ is a pro-open immersion.

(vi) The class of pro-open immersions is closed under fpqc descent. Namely, if $i\:\calU\to X$ is a morphism of algebraic spaces and $i\times_X X'$ is a pro-open immersion for an fpqc covering $X'\to X$ then $i$ itself is a pro-open immersion.
\end{prop}
\begin{proof}
The assertion (i) follows by applying the universal property in the definition to points $\Spec(k)\to X$ with image in $U$. Then the assertions (ii), (iv), (v) and the property of being a monomorphism follow from the universal property, e.g., since $U=\cap_\alp |\calU_\alp|$, we have an isomorphism of functors $h_U\toisom\lim_\alp h_{\calU_\alp}$ regardless of their representability. Since $\calU_\alp$ represents $h_{\calU_\alp}$, this implies (ii).

To prove (iii), it remains to establish flatness, which is local on the target. Thus we may assume that $X$ is an affine scheme. Since flatness is local on the source too, it is sufficient to show that if $u\in U$ then $\calU\times_XX_u\to X_u$ is flat, where $X_u$ is the localization of $X$ at $u$. But the latter is a surjective pro-open immersion by (iv), hence an isomorphism, and in particular flat.

It remains to prove (vi). Set $\calU':=\calU\times_XX'$, $X'':=X'\times_XX'$ and $\calU'':=\calU\times_X X''$. By (iv), $\calU''\to X''$ is a pro-open immersion. For any morphism $f\:Y\to X$ with $f(|Y|)\subseteq|\calU|$, its base changes $f'\:Y'\to X'$ and $f''\:Y''\to X''$ land in $|\calU'|$ and $|\calU''|$ respectively. Hence $f'$ and $f''$ induce morphisms $g'\:Y'\to\calU'$ and $g''\:Y''\to\calU''$, and these morphisms descend to an $X$-morphism $g\:Y\to\calU$ by fpqc descent. Therefore, $i$ is a pro-open immersion.
\end{proof}

\begin{theor}\label{protop}
Pro-open immersions of algebraic spaces are schematic.
\end{theor}
\begin{proof}
By Rydh's theorem \cite[Tag:0B8A]{stacks}, any flat monomorphism of algebraic spaces is schematic, hence so are pro-open immersions by Proposition~\ref{proprop}(iii).
\end{proof}

\subsection{Topology}
Our next aim is to characterize pro-open immersions as flat mono\-morphisms that satisfy certain topological restrictions. Recall, that a topological space is called {\em locally quasi-compact} if any point admits a fundamental system of quasi-compact neighborhoods \cite[Tag:08ZQ]{stacks}.

\subsubsection{Topological embedding}
First, we prove that a pro-open immersion is a topological embedding if and only if the corresponding pro-open subspace is locally quasi-compact.

\begin{theor}\label{topth}
Let $X$ be an algebraic space, $U\subseteq|X|$ a pro-open subspace, and $\calU\to X$ the corresponding pro-open immersion. Then $U$ is locally quasi-compact with respect to the induced topology (e.g., $U\subseteq|X|$ is retrocompact) if and only if the map $|\calU|\to U$ is a homeomorphism.
\end{theor}
\begin{proof}
Since any algebraic space $\calU$ is locally quasi-compact, the inverse implication is clear. Let us prove the direct one. Since the claim is local on $X$ and $U$, we may assume $X$ and $U$ are qcqs. Recall that $j\:|\calU|\to U$ is bijective by Proposition~\ref{proprop}(i), and the morphism $i\:\calU\to X$ is flat by Proposition~\ref{proprop}(iii). Hence $j$ is an $S$-homeomorphism.

By \cite[Tag:0A4G]{stacks}, a qcqs algebraic space is spectral, hence $|X|$ is spectral. We claim that $|\calU|$ is spectral too. By \S\ref{factsec}(i), $U$ is pro-constructible, and by \S\ref{factsec}(iv) there exists a morphism $h\:T\to X$ such that $T$ is affine and $h(T)=U$. By the universal property, $h$ factors through $\calU$, and we obtain a surjective morphism $T\to\calU$. In particular, $\calU$ is qcqs, and hence spectral.

Note that $\calU$ is compact in the constructible topology and $U$ is Hausdorff in the topology induced from the constructible topology of $X$, hence the continuous bijection $j\:\calU\to U$ is a homeomorphism with respect to these topologies. It then follows from \S\ref{factsec}(iii) that $j$ is also a homeomorphism with respect to the Zariski topologies.
\end{proof}

\begin{rem}
(i) If a pro-open subspace $U$ is locally quasi-compact then Theorem~\ref{topth} allows, by a slight abuse of language, to identify $U$ with $\calU$. So, we will view $\calU$ as the topological space $U$ equipped with an additional structure defined uniquely by $U$ and $X$, and for brevity we will not usually distinguish $U$ and $\calU$.

(ii) We will provide in \S\ref{wildproopensec} an example of a pro-open subspace, which is not locally quasi-compact. All our other results will concern with the locally quasi-compact case.
\end{rem}

\subsubsection{Relation to flat monomorphisms}
Now we can characterize pro-open immersions as follows.

\begin{theor}\label{critproopen}
Let $f\:Y\to X$ be a morphism of algebraic spaces.

(i) If $f$ is a flat monomorphism and $|f|$ is a topological embedding then $f$ is a pro-open immersion.

(ii) Assume that $f$ is quasi-compact. Then $f$ is a pro-open immersion if and only if $f$ is a flat monomorphism.
\end{theor}
\begin{proof}
(ii) The direct implication follows from Proposition~\ref{proprop}(iii). To prove the inverse implication set $U:=f(Y)$. Then $U$ is generizing by the flatness of $f$. Given a morphism $h\:Z\to X$ with image in $U$, the base change $Z\times_XY\to Z$ is a quasi-compact surjective flat monomorphism, and hence an isomorphism by \cite[Tag:0B8C]{stacks}. Thus, $h$ lifts to $Y$ as needed.

(i) Pick an open covering $Y=\cup_i Y_i$ such that each $Y_i$ is qcqs, and set $U:=f(Y)$ and $U_i:=f(Y_i)$. Then $U=\cup U_i$ is an open covering since $|f|$ is a topological embedding. Let $g\:Z\to X$ be a morphism with image in $U$. Then $Z_i:=g^{-1}(U_i)\subseteq Z$ is open, and the surjective flat monomorphisms $Z_i\times_XY=Z_i\times_XY_i\to Z_i$ are quasi-compact, thus isomorphisms by \cite[Tag:0B8C]{stacks}. Hence $Z\times_XY\to Z$ is an isomorphism, and $f$ is a pro-open immersion.
\end{proof}

\subsubsection{Faithfully flat descent}\label{bassec}
In the case of quasi-compact pro-open immersions one can strengthen Proposition~\ref{proprop}(vi) as follows.

\begin{prop}\label{qcproprop}
Let $i\:Y\to X$ be a quasi-compact morphism, and $X'\to X$ be a flat surjective morphism. If $i\times_X X'$ is a pro-open immersion then so is $i$.
\end{prop}
\begin{proof}
Note that flatness is preserved by an arbitrary faithfully flat descent. Also, the same is true for the property of being a monomorphism because morphisms $p_{1,2}\:Z\to Y$ are equal if and only if their base changes $p_{1,2}\times_XX'$ are equal. Thus, the assertion follows from Theorem~\ref{critproopen}(ii).
\end{proof}

\subsection{Pathologies}
In this section we collect various examples of pathological generizing sets and pro-open subspaces and immersions to justify the restrictions we had and will have to impose in our results.

\subsubsection{Flat monomorphisms}
Here is a typical example of a flat monomorphism which is not a topological embedding. Let $C$ be an integral curve with infinitely many closed points. Consider the curve $C'$ obtained by gluing all localizations $C_x=\Spec(\calO_{C,x})$ along their generic points. Then the natural morphism $f\:C'\to C$ is a bijective flat monomorphism, which is not a homeomorphism. Actually $C'$ is obtained from $C$ by replacing the topology of $C$ by the $S$-topology, and ``preserving" the structure sheaf. By Theorem~\ref{topth} $f$ is not a pro-open immersion. This can be also observed directly since the identity $C=C$ does not factor through $C'$.
Let $X$ be a reduced zero-dimensional qcqs scheme with a non-discrete point $x\in X$; we will give a few examples of such $X$ below. Then $\calO_{X,x}=k(x)$, hence the closed immersion $i\:x\into X$ is a pro-open immersion of finite type, which is not an open immersion. In particular, $V=x$ is not open.

\subsubsection{A quasi-compact generizing subset which is not a pro-open subspace}\label{notspacesec}
Take an affine plane $Y=\bfA^2_k$ over a field $k$ and let $V$ be obtained from $X$ by removing all closed points. We claim that $V$ is not a pro-open subspace. Indeed, if $V$ corresponds to a pro-open immersion $\calV\to X$ then $V=|\calV|$ by Theorem~\ref{topth}. One can also see this directly since $\calV$ is covered by two pro-open subspaces $\Spec(k(x)[y])$ and $\Spec(k(y)[x])$, which are the generic fibers of the projections of $Y$ onto the axis. Thus, any open subscheme $\calU\subseteq\calV$ is obtained by removing a finite set $Z$ of closed points. To get a contradiction it suffices to show that no such $\calU$ is an affine scheme. But $\calU$ is strictly smaller than the spectrum of $A=\calO_\calV(\calU)=\cap_{y\in\calU}\calO_{Y,y}$, since $A$ is the ring of functions on the complement to the Zariski closure of $Z$ in $Y$.

\subsubsection{Finite type examples}\label{examsect}
Given a morphism of schemes $f\:Y\to X$, let $V$ be the set of points $x\in X$ such that $Y\times_X\Spec(\calO_{X,x})=\Spec(\calO_{X,x})$ (it is the pro-open locus of $f$ as will be defined later). We will show that even when $f$ is of finite type the generizing set $V$ can be nasty.

\begin{exam}\label{nonreducedexam}
(i) Let $X$ be a scheme with reduction $Y=\Spec(\bfZ)$ and such that each closed point supports an embedded component. For example, one can take $X=\Spec(A)$ for $A=\bfZ[x_2,x_3,x_5,\dots]/(x_2^2,2x_2,x_3^2,3x_3,\dots)$. If $Y\to X$ is the reduction morphism then $V=\Spec(\bfQ)$ is not open in $X$ and $Y$.

(ii) Moreover, in the same manner one can take an affine plane $Y=\bfA^2_k$ over a field $k$ and construct $X$ by inserting embedded components at all closed points of $Y$. Then $V$ is obtained from $X$ by removing all closed points and it is not even a pro-open subspace by \S\ref{notspacesec}.
\end{exam}

Restricting to reduced schemes does not improve the situation much:

\begin{exam}\label{otherbadex}
Let $X$ be a reduced zero-dimensional compact scheme with a non-discrete point $x\in X$ (see \S\ref{zerodimsec} below). Then $\calO_{X,x}=k(x)$, hence the closed immersion $i\:x\into X$ is a pro-open immersion of finite type, which is not an open immersion. In particular, $V=x$ is not open.
\end{exam}

\subsubsection{Non-discrete zero-dimensional schemes}\label{zerodimsec}
Here are two classical examples of a reduced zero-dimensional compact scheme $X$ with a non-discrete point $x\in X$.

(i) Consider an infinite product of fields $A=\prod_{i\in I} k_i$. Then $X=\Spec(A)$ is a compact topological space naturally homeomorphic to the Stone-\v{C}ech compactification $\hatI$ of the discrete set $I$. Clearly, $X$ contains many non-discrete points: these are precisely the points of $\hatI\setminus I$, i.e. the non-principal ultrafilters on $I$.

\begin{rem}
The first author learned this example from B. Conrad in a slightly different context: $i\:x\into X$ is an essentially \'etale morphism of finite type which is not \'etale. We also remark that $i$ is flat and of finite type but not of finite presentation, and note for the sake of comparison that if $X$ is an arbitrary integral scheme then any flat morphism $Y\to X$ of finite type is also of finite presentation by \cite[Corollary~3.4.7]{RG}.
\end{rem}

(ii) This example goes back to J.-P. Olivier \cite{O} and was suggested to us by D. Rydh. It is known that there exists an affine bijective monomorphism $X=\Spec(B)\to\Spec(\bfZ)$ such that $X$ is universally flat (i.e., any $B$-module is flat), and the topology on $X$ is the constructible topology of $\Spec(\bfZ)$. In particular, $X$ is a zero dimensional qcqs scheme with a non-discrete point $\Spec(\bfQ)$ and discrete points $\Spec(\bfF_p)$. Note also that $B$ can be described explicitly as $B=\bfZ[T_2,T_3,T_5,\dotsc]/I$, where $I$ is the ideal generated by $(pT_p-1)p$ and $(pT_p-1)T_p$ for all primes $p\in\bfZ$.

\subsubsection{A pro-open immersion which is not a topological immersion}\label{wildproopensec}
We did not know if such immersions exist. The following example was suggested by the anonymous referee: Let $X=\Spec(A)$ and $x\in X$ be as in \S\ref{zerodimsec}(i). Thus $|X|=\hatI$ and $x\in\hatI\setminus I$. Any subset of $X$ is generizing, and we set $U=I\cup\{x\}$. Note that $x$ is a non-discrete point of $U$. Set $$\calU=\left(\coprod_{i\in I}\Spec(k_i)\right)\coprod\Spec(k(x)),$$ then the space $|\calU|$ is discrete, and hence $|\calU|$ is not homeomorphic to $U$.

We claim that the natural morphism $j\:\calU\to X$ mapping $\calU$ onto $U$ is a pro-open immersion. Indeed, let $f\:T\to X$ be a morphism for which $f(T)\subseteq U$, and let us show that $f$ lifts to $\calU$. Since $j$ is a monomorphism the lifting, if exists, is unique. Thus, we may assume that $T$ is quasi-compact, and hence $\oY:=f(T)\subseteq U$ is compact. Since compact sets are closed, $\oY\subset\hatI$ contains the closure of $Y:=\oY\cap I$ in $\hatI$, which is the Stone-\v{C}ech compactification of $Y$. Thus, $\oY\setminus Y$ contains the set of all non-principal ultrafilters of $Y$, and if $Y$ is inifinite then there are infinitely (even uncountably) many of those. Since $\oY\setminus Y=\{x\}$ we obtain that $Y$ is finite, and hence $\oY$ is finite too. In particular, $\oY$ is discrete, hence the restriction of $j$ over $\oY$ is an isomorphism and we obtain the required lifting of $f$.

Notice that by Theorem~\ref{topth}, it follows that $U$ is not locally quasi-compact at $x$, which in this case can easily be verified directly.

\subsection{Relation to results of Raynaud}
In \cite{Ray}, Raynaud studied extensively (quasi-compact) flat monomorphisms of schemes. We extend his results to pro-open immersions of algebraic spaces.

\subsubsection{Affine pro-open immersions}
We start with the affine case.

\begin{prop}\label{procritprop}
Let $X$ be a qcqs algebraic space, $X'\to X$ a flat quasi-compact presentation, and $U\subseteq|X|$ a quasi-compact generizing subset. Then $U$ is a pro-open subspace with an affine pro-open immersion $\calU\to X$ if and only if for any point $x'\in X'$, the preimage of $U$ in $\Spec(\calO_{X',x'})$ is an affine pro-open subscheme.
\end{prop}
\begin{proof}
The direct implication follows from Proposition~\ref{proprop}(iv). For the inverse implication, notice that the preimage of $U$ in $X'$ is quasi-compact by \S\ref{factsec} (i) since $X'\to X$ is quasi-compact. Thus, by Proposition~\ref{proprop}(vi), we may assume that $X=X'$ is a scheme. The assertion now follows from \cite[Proposition~2.4(B)]{Ray}.
\end{proof}

\begin{prop}
Any quasi-compact pro-open immersion of algebraic spaces $\calU\into X$ is quasi-affine.
\end{prop}
\begin{proof}
By descent, we may assume that $X$ is a scheme. Then $U$ is also a scheme by Theorem~\ref{protop} and the assertion is covered by \cite[Proposition~1.5]{Ray}.
\end{proof}

\begin{prop}
Let $U$ be a generizing subset of a scheme $X$ such that the induced topology on $U$ is locally quasi-compact. Then $U$ is a pro-open subspace if and only if the locally ringed space $Y:=(U,\calO_X|_U)$ is a scheme, and then $Y\to X$ is the pro-open immersion.
\end{prop}
\begin{proof}
The direct implication is immediate from the definition. Conversely, if $U$ is pro-open and $\calU\to X$ is the corresponding pro-open immersion then $\calU$ is a scheme by Theorem~\ref{protop}. By Theorem~\ref{topth}, $\calU\to U$ is a homeomorphism, and $\calO_{X,x}=\calO_{\calU,x}$ because $\calU_x\to X_x$ is a surjective pro-open immersion and hence an isomorphism. Thus, $Y=\calU$ is a scheme. 
\end{proof}

\subsection{Morphisms of finite type}\label{ftsec}
Next, we study relations between pro-open immersions and morphisms of finite type.

\subsubsection{Pro-open locus}
By a {\em pro-open locus} of a morphism $f\:Y\to X$ whose target is a scheme we mean the subset $V\subseteq X$ consisting of points $x\in X$ such that $Y\times_X X_x=X_x$ where $X_x=\Spec(\calO_{X,x})$. Clearly, $V$ is generizing and the map $f^{-1}(V)\to V$ is bijective. By examples of \S\ref{examsect}, this is all one can say about $V$ even when $f$ is of finite type.

Pro-open loci are easily seen to be compatible with flat base changes, i.e., if $g\:X'\to X$ is a flat morphism of schemes then $g^{-1}(V)$ is the pro-open locus of $f\times_XX'$. It follows that given an arbitrary morphism of algebraic spaces $f\:Y\to X$, a presentation $g\:X'\to X$ and the pro-open locus $V'$ of $f\times_XX'$, the subset $V=f(V')$ of $X$ is independent of the presentation. We call $V$ the {\em pro-open locus} of $f$. Pro-open loci are compatible with arbitrary flat base changes, as can be easily deduced from the scheme case.

\subsubsection{Openness of pro-open locus}
In order to obtain some control on $V$ one has to strengthen the naive assumption that $f$ is of finite type, and we suggest two ways to do this. The standard approach is to assume that $f$ is of finite presentation. This assumption is too strong for our applications. A less standard assumption which will work fine in our context is to assume that $f$ is schematically dominant. Note that neither assumption is satisfied in Examples~\ref{nonreducedexam} and \ref{otherbadex}.

\begin{prop}\label{proopenlocus}
Assume that $f\:Y\to X$ is a morphism such that at least one of the following conditions is satisfied:
\begin{itemize}
\item[(i)] $f$ is of finite presentation,

\item[(ii)] $f$ is schematically dominant and of finite type.
\end{itemize}
Then the pro-open locus $V\subset X$ is an open subspace and the morphism $V\times_XY\to V$ is an isomorphism.
\end{prop}
\begin{proof}
If $X'\to X$ is an \'etale morphism then the base change $f'\:Y\times_XX'\to X'$ satisfies the same condition (i) or (ii) as $f$ and the pro-open locus of $f'$ is the preimage of $V$. Thus, the claim is \'etale-local on $X$, and we may assume that $X$ is a qcqs scheme. Fix a point $x\in V$ and let $y\in Y$ be its preimage. It suffices to show that $V$ contains a neighborhood $U$ of $x$ such that $U\times_XY=U$.

(i) Let $U_\alp$ be the family of affine neighborhoods of $x$. Since $X_x=\Spec(\calO_{X,x})$ is the limit of $U_\alp$ and $X_x\times_XY=X_x$, it follows by approximation that already for a large enough $\alp$ we have that $U_\alp\times_XY=U_\alp$.

(ii) By approximation of morphisms we can realize $Y$ as a closed subspace of an algebraic space $Z$ of finite presentation over $X$. Moreover, $Y=\lim_\alp Z_\alp$, where $Z_\alp$ are finitely presented closed subspaces of $Z$. Then $\calO_{X,x}=\calO_{Y,y}$ is the filtered colimit of the quotients $\calO_{Z_\alp,y}$ of $\calO_{Z,y}$. It follows that already for some $\alp$ the equality $\calO_{Z_\alp,y}=\calO_{X,x}$ holds and after replacing $Z$ by $Z_\alp$ we may assume that $\calO_{X,x}=\calO_{Z,y}$. By case (i), after replacing $X$ by a neighborhood of $x$ we may assume that $Z\to X$ is an isomorphism. Since $f$ is schematically dominant this implies that the closed immersion $Y\into Z$ is an isomorphism, and hence $Y=X$.
\end{proof}

\subsubsection{Pro-open immersions of finite type}
We start with describing certain situations in which a pro-open immersion is open.

\begin{cor}\label{factorcor}
Let $f\:Y\to X$ be a morphism of algebraic spaces. Then,

(i) $f$ is an open immersion if and only if it is a pro-open immersion locally of finite presentation.

(ii) Assume that $f$ is schematically dominant. Then $f$ is a quasi-compact open immersion if and only if it is a pro-open immersion of finite type.
\end{cor}
\begin{proof}
Only inverse implications need a proof. In claim (ii) this follows directly by applying Proposition~\ref{proopenlocus}(ii) to $f$. In claim (i), we shall prove that if a pro-open $f$ is locally of finite presentation then $f(Y)$ is open. This is local on $Y$, hence we may assume that $f$ is of finite presentation and use Proposition~\ref{proopenlocus}(i).
\end{proof}

Now, we can characterize pro-open immersions of finite type.

\begin{prop}\label{ftproopen}
Let $f\:Y\to X$ be a morphism of algebraic spaces. Then the following conditions are equivalent:

(i) $f$ is a pro-open immersion of finite type,

(ii) $f$ is a flat monomorphism of finite type,

(iii) $f$ is a quasi-compact flat locally closed immersion.
\end{prop}
\begin{proof}
The equivalence (i)$\Longleftrightarrow$(ii) follows from Theorem~\ref{critproopen}(ii) and the implication (iii)$\implies$(ii) is obvious. Assume that (i) is satisfied and let $Z$ be the schematic image of $f$. Then $g\:Y\to Z$ is a pro-open immersion by Lemma~\ref{lem} below and hence $g$ is an open immersion by Corollary~\ref{factorcor}(ii). Thus, $f$ is a locally closed immersion and we obtain (iii).
\end{proof}

\begin{lem}\label{lem}
Assume that $f\:Y\to X$ is a pro-open immersion and $Z$ is the schematic image of $f$. Then $g\:Y\to Z$ is a pro-open immersion.
\end{lem}
\begin{proof}
Note that $Y=Y\times_ZZ=Y\times_ZZ\times_XZ=Y\times_XZ$, hence $g$ is a base change of $f$ and it remains to use Proposition~\ref{proprop} (iv).
\end{proof}

As a corollary we can characterize pro-open immersions of finite type in purely topological terms.

\begin{cor}
Let $X$ be an algebraic space with a generizing subset $U\subseteq|X|$. Then $U$ corresponds to a pro-open immersion of finite type if and only if $U$ is locally closed and retrocompact.
\end{cor}
\begin{proof}
The direct implication follows from Proposition~\ref{ftproopen}. The inverse implication will follow from Proposition~\ref{ftproopen}, once we prove that $U$ underlies a flat locally closed immersion $\calU\into X$. By descent, it suffices to prove the latter claim for schemes. For a closed $U$, this was proved in \cite[Tag:04PW]{stacks}, and the general case follows immediately.
\end{proof}


\begin{cor}\label{strangecor}
Let $X$  be a qcqs algebraic space with a pro-open subspace of finite type $Y\into X$. Assume that there exists a dense quasi-compact open subset $U\subseteq X$ such that $Y\cap U$ is open. Then $Y$ is open.
\end{cor}
\begin{proof}
By Proposition~\ref{ftproopen}, $Y$ is closed in an open subspace $X'\into X$. Since $Y$ is quasi-compact, $X'$ can be chosen to be quasi-compact. Thus, after shrinking $X$, we may assume that $Y\subseteq X$ is closed.

By the assumption, $U_1:=Y\cap U$ is clopen in $U$. Hence $U$ is the disjoint union of clopen subsets $U_1$ and $U_2:=U\setminus Y$, and by quasi-compactness of $U$, both $U_i$ are quasi-compact. By \cite[Tag:0903]{stacks}, any point in the closure $\oU_2\subseteq X$ is a specialization of a point of $U_2$. Thus, $\oU_2$ is disjoint from $Y$ since $Y$ is generizing. The closure $\oU_1\subseteq X$ is contained in $Y$ since $Y$ is closed. However $U$ is dense, thus $\oU_1=Y$ and $\oU_2=X\setminus Y$. It follows that $Y$ is open in $X$.
\end{proof}

\section{Pr\"ufer algebraic spaces and pairs}\label{pairsec}

\subsection{Definitions and formulations of main results}\label{defmainsec}

\subsubsection{Semivaluation rings}\label{semisec}
By a {\em semivaluation ring} we mean a pair of rings $(B,A)$ such that $B$ is local with maximal ideal $m$ and $A\subset B$ is the preimage of a valuation ring $R$ of $k=B/m$. Often we will refer to $A$ as the semivaluation ring and will say that $B$ is its {\em semifraction} ring. Note that $m\subset A$, $B=A_m$, and $m$ is divisible by any element of $A\setminus m$. In addition, we claim that $A$ is a valuation ring if and only if $B$ is a valuation ring. Indeed, this follows easily from the fact that $A$ is a valuation ring if and only if it is integral and for any  $f\in\Frac(A)\setminus A$ one has $f^{-1}\in A$.

\begin{rem}\label{semirem}
(i) We use the word ``semivaluation" because the valuation of $R$ induces a semivaluation on $B$ with kernel $m$, which is unique up to equivalence, and $A$ is the ring of integers of this semivaluation.

(ii) In abstract commutative algebra, a semivaluation (ring) is often called a local Manis valuation (ring), see \cite[Ch. I]{KZ}. We prefer the terminology of \cite{temrz}, where semivaluation rings were used to study relative Riemann-Zariski spaces.

(iii) Note that $A\toisom B\times_kR$; so $\Spec(A)$ is the affine Ferrand pushout of the closed valuation subscheme $\Spec(R)$ and the localization $\Spec(B)$ along the point $\Spec(B/m)$ (cf. \cite{push}).
\end{rem}

\subsubsection{Pr\"ufer algebraic spaces}
An integral qcqs algebraic space $X$ is called {\em Pr\"ufer} if any modification $X'\to X$ is {\em trivial}, i.e., an isomorphism. An algebraic space is called {\em Pr\"ufer} if it is a finite disjoint union of integral Pr\"ufer spaces. In particular, Pr\"ufer spaces are normal.

\begin{rem}\label{rem111i}
The terminology is motivated by the fact that affine integral Pr\"ufer schemes are nothing but the spectra of Pr\"ufer domains, see \S\ref{connsec} below.
\end{rem}

\subsubsection{Separated Pr\"ufer spaces}
Examples of non-separated Pr\"ufer spaces which are not schematic will be given in \S\ref{examsec}. On the other hand, we have the following result.

\begin{prop}\label{sepprufprop}
Let $X$ be a separated quasi-compact integral algebraic space. Then there exists a modification $Z\to X$ such that $Z$ is a scheme. In particular, any separated Pr\"ufer algebraic space $X$ is a scheme.
\end{prop}
\begin{proof}
Let $\eta\in X$ be the generic point. By approximation of morphisms \cite[Theorem D]{rydapr}, there exists an affine morphism $f\:X\to X_0$, where $X_0$ is separated and of finite type over $\bfZ$. Replacing $X_0$ with the schematic image of $X$ we may also assume that it is irreducible with generic point $\veps=f(\eta)$. By Chow's lemma \cite[Ch. IV, Theorem~3.1]{Knu} there exists a modification $Y_0\to X_0$ such that $Y_0$ is quasi-projective over $\bfZ$. Then $Y_0$ is a scheme, and since $Y:=Y_0\times_{X_0}X$ is affine over $Y_0$, it is a scheme too. Since $Y_0\to X_0$ is an isomorphism over $\veps$, the morphism $g\:Y\to X$ is an isomorphism over $\eta$ and hence the schematic closure $Z$ of $g^{-1}(\eta)$ in $Y$ is a modification of $X$. Obviously, $Z$ is a scheme.

If $X$ is Pr\"ufer then by the above every connected component $X_i$ of $X$ possesses a modification $Z_i$ which is a scheme. But $Z_i=X_i$ by the Pr\"uferness.
\end{proof}

\begin{rem}\label{rem111}
Although it is easy to describe Pr\"ufer schemes using the theory of Pr\"ufer domains, it is more difficult to descend this description to algebraic spaces (in the non-separated case). The main difficulty is that not any closed subscheme of the non-Pr\"ufer locus of a scheme $X$ (which is a subset of $X$ closed under specialization) may serve as the modification locus of a modification $X'\to X$. For example, assume that $A$ is a semivaluation ring from \S\ref{semisec} such that $R\subsetneq k$. Then it can be easily checked that $S=\Spec(A)$ does not admit modifications that modify only the closed point. On the other hand, $S$ is not Pr\"ufer when $B$ is not a valuation ring. In order to gain a better control on the modification loci we are going to introduce the notion of a Pr\"ufer pair.
\end{rem}

\subsubsection{Pr\"ufer pairs}
Assume that $X$ is a qcqs algebraic space with a quasi-compact schematically dense generizing subset $U\subseteq|X|$. We say that $(X,U)$ is a {\em Pr\"ufer pair} if any $U$-modification of $X$ is trivial.

\begin{rem}
(i) An integral qcqs $X$ is Pr\"ufer if and only if the pair $(X,\eta)$ is Pr\"ufer, where $\eta$ is the generic point of $|X|$.

(ii) Let $S=\Spec(A)$ be as in Remark~\ref{rem111}, and consider its pro-open subset $U=\Spec(B)$. Then it is easy to see (and will be proved in Theorem~\ref{prufpairth}) that the pair $(S,U)$ is Pr\"ufer, i.e., any non-trivial modification of $S$ modifies $U$.
\end{rem}

\subsubsection{Characterization of Pr\"ufer pairs}
Now, we can formulate our main result on Pr\"ufer pairs. Its proof occupies the rest of \S\ref{defmainsec} and the whole \S\ref{schemesec}.

\begin{theor}\label{prufpairth}
Let $X$ be a qcqs algebraic space with a quasi-compact schematically dense generizing subset $U\subset|X|$. The following conditions are equivalent:
\begin{enumerate}
\item
There exists an \'etale presentation $f\:X'\to X$ such that for any point $x'\in X'$ there exists a unique minimal generalization $u'\in U':=f^{-1}(U)$, and the pair $(\calO_{U',u'},\calO_{X',x'})$ is a semivaluation ring.
\item
For any scheme $X'$ with an \'etale morphism $f\:X'\to X$ and a point $x'\in X'$ there exists a unique minimal generalization $u'\in U':=f^{-1}(U)$, and the pair  $(\calO_{U',u'},\calO_{X',x'})$ is a semivaluation ring.
\item
A $U$-admissible morphism $f\:Y\to X$ is flat if and only if it is $U$-flat, i.e., flat at any $y\in f^{-1}(U)$.
\item
If $f\:Y\to X$ is a separated quasi-compact $U$-admissible morphism such that $\oU\times_XY\to\oU$ is a pro-open immersion for a pro-open subspace $\oU$ with $U\subseteq|\oU|$ then $f$ is a pro-open immersion.

\item
Any $U$-quasi-modification $Y\to X$ is an open immersion.
\item
$X$ admits no non-trivial $U$-modifications, i.e., $(X,U)$ is a Pr\"ufer pair.
\item
$X$ admits no non-trivial $U$-admissible blow ups.
\item
Any finitely generated $U$-trivial ideal $\calI\subset\calO_X$ is invertible.
\end{enumerate}
In addition, if these conditions hold then $U$ is a pro-open subspace and the pro-open immersion $i\:U\into X$ is affine.
\end{theor}

\subsubsection{First implications}\label{implsec}
In this section we establish easy implications between the conditions (1)--(8). The implications $(2)\Ra(1)$ and $(5)\Ra(6)\Ra(7)\Lra(8)$ are obvious. Let us show that $(3)\Ra(4)$. Assume that $f$ satisfies the assumption of (4). Then $f$ is flat by (3) and Proposition~\ref{proprop}(iii). Since the restriction of $f$ onto the schematically dense pro-open subspace $\oU\times_XY$ is a pro-open immersion, $f$ itself is a pro-open immersion by the following lemma:

\begin{lem}\label{prolem}
Assume that $f\:Y\to X$ is a flat separated morphism of qcqs algebraic spaces and $U\into Y$ is a schematically dominant pro-open immersion such that the composition $U\to X$ is a pro-open immersion. Then $f$ is a pro-open immersion.
\end{lem}
\begin{proof}
Both morphisms in the sequence $U\times_XU\to U\times_XY\to Y\times_XY$ are base changes of $U\to Y$. Hence the composition is a schematically dominant pro-open immersion by Lemma~\ref{flatbaselem} and Proposition~\ref{proprop}. Therefore, the diagonal closed immersion $Y\to Y\times_XY$ is schematically dominant, and hence an isomorphism. Thus, $f$ is a monomorphism, hence a pro-open immersion by Theorem~\ref{critproopen}.
\end{proof}

Next, we show that $(4)\Ra(5)$. Let $f\:Y\to X$ be a $U$-quasi-modification. Then it is a pro-open immersion by (4). By definition and the quasi-compactness of $U$, there exists a quasi-compact open $V\subseteq X$ containing $U$ such that $f$ is a $V$-quasi-modification. Thus, $f$ is an open immersion by Corollary~\ref{strangecor}.

Finally, we note that if (1) is satisfied then $X'_{\succeq x'}\cap U'=X'_{\succeq u'}$, and therefore Proposition~\ref{procritprop} implies that $U\to X$ is an affine pro-open immersion. Our proof of the remaining claims will proceed by establishing the following implications: $(8)\Ra(1)\Ra(2)\Ra(3)$. This will be done in \S\ref{schemesec}, but let us first deduce some corollaries and connect this theory to certain well known results about rings.

\subsubsection{Pr\"ufer algebraic spaces}
Taking $U$ to be the generic point of an integral algebraic space one immediately obtains a similar description of Pr\"ufer spaces:

\begin{cor}\label{prufspaceth}
Let $X$ be an integral qcqs algebraic space. The following conditions are equivalent:
\begin{enumerate}
\item
There exists an \'etale presentation $X'\to X$ such that for any point $x'\in X'$ the local ring $\calO_{X',x'}$ is a valuation ring.
\item
For any scheme $X'$ with \'etale morphism $X'\to X$ and a point $x'\in X'$ the local ring $\calO_{X',x'}$ is a valuation ring.
\item
A morphism $Y\to X$ is flat if and only if it is $\eta$-admissible, where $\eta$ is the generic point of $X$.
\item
If $f\:Y\to X$ is a separated quasi-compact morphism such that $Y$ is integral, and $f$ induces an isomorphism of the generic points then $f$ is a pro-open immersion.
\item
Any quasi-modification $Y\to X$ is an open immersion.
\item
$X$ admits no non-trivial modifications, i.e., $X$ is a Pr\"ufer space.
\item
Any non-empty blow up of $X$ is an isomorphism.
\item
Any non-zero finitely generated ideal $\calI\subset\calO_X$ is invertible.
\end{enumerate}
\end{cor}

\subsubsection{\'Etale-locality of Pr\"uferness}
As an immediate corollary of conditions (1) and (2) of Theorem~\ref{prufpairth} we obtain the following important result:

\begin{cor}\label{etpaircor}
Let $f\colon X'\to X$ be an \'etale morphism between qcqs algebraic spaces, $U\subset|X|$ a quasi-compact schematically dense generizing subset, and $U':=f^{-1}(U)$. If $(X,U)$ is a Pr\"ufer pair then $(X',U')$ is a Pr\"ufer pair, and the converse is true whenever $f$ is surjective. In particular, if $X$ is Pr\"ufer then $X'$ is Pr\"ufer and the converse is true for surjective $f$.
\end{cor}

\subsubsection{Connection to the theory of Pr\"ufer rings}\label{connsec}
If $X=\Spec(A)$ is an integral affine scheme then $X$ is Pr\"ufer if and only if $A$ is a Pr\"ufer domain. The rings of the latter type were intensively studied in abstract commutative algebra. In particular, there are many other conditions equivalent to Pr\"uferness of $A$ (or $X$). For example, $14$ such conditions are listed in \cite[Ch. VII, \S2, Exercise 12]{Bou}, including our (1) and (8) from Corollary~\ref{prufspaceth}. Similarly, if $X=\Spec(R)$ is affine then a pair $(X,U)$ is Pr\"ufer if and only if $U=\Spec(A)$ is affine, $R\into A$ and $R$ is an $A$-Pr\"ufer ring. To see this, one can use equivalent condition (5) of Theorem~\ref{prufpairth} and equivalent condition (11) of \cite[Ch I, Theorem~5.2]{KZ}.

Thus, in the affine case our theory describes a classical algebraic object from the geometric point of view, though our geometric definition of Pr\"uferness does not appear among the 14 equivalent algebraic definitions of \cite{Bou}. The advantage of the geometric approach is that it makes sense for global objects and for stacks. In principle, one could use the algebraic theory as a slight shortcut for proving Theorem~\ref{prufpairth}, but in order to give the geometric arguments, that we plan to generalize to stacks in further works, we prefer to minimize the use of the theory of Pr\"ufer rings.

\subsection{Proofs}\label{schemesec}
In the first two sections of \S\ref{schemesec} we will prove Theorem~\ref{prufpairth} in the case when $X$ is a local scheme\footnote{This is the case, where one could shorten some arguments by quoting results on Pr\"ufer pairs of rings.}. The case when $X$ is a scheme is deduced in \S\ref{schemecasesec} using the fact that one can extend ideals from open subschemes. Finally, the general case is done in \S\ref{genalgsec}, and it requires a more subtle argument because one cannot descend arbitrary ideals through \'etale presentations $X'\to X$.

\subsubsection{\'Etale-locality}\label{indepsec}
In this section we prove the implication $(1)\Rightarrow (2)$, which follows from the following:

\begin{lem}\label{etpairlem}
Assume that $f\:X'\to X$ is an essentially \'etale morphism of local schemes, where $X=\Spec(A)$ and $X'=\Spec(A')$. Assume, in addition, that $U\subset X$ is a generizing subset and $U'=f^{-1}(U)$.

(i) If $U=\Spec(B)$ where $A\into B$ and $(B,A)$ is a semivaluation ring then $U'=\Spec(B')$ where $A'\into B'$ and $(B',A')$ is a semivaluation ring.

(ii) The converse of (i) is true whenever $f$ is surjective, i.e., the homomorphism $A\to A'$ is local.
\end{lem}
\begin{proof}
We will treat both cases in the same fashion. First, we claim that in both cases we may assume that $U=\Spec(B)$ and $U'=\Spec(B')$ are local schemes with closed points $x$ and $x'$, respectively. In case (ii), $U=f(U')=f(X'_{\succeq x'})=X_{\succeq f(x')}$ since $f$ is generizing (cf. \S\ref{genersec}). In case (i) we have $U=X_{\succeq x}$. We may assume that $f^{-1}(x)\neq\emptyset$ as otherwise $U'=X'$ and the lemma follows. Let $T=\ox$ be the closure of $x$ in $X$ with the reduced scheme structure. Then $T':=T\times_XX'\subseteq X'$ is closed, and hence a local scheme. Since $T$ is a valuation scheme, the same is true for $T'$ by Lemma~\ref{lem:etcovofvalsch}; in particular, $f^{-1}(x)=\{x'\}$.

We claim that the inclusion $X'_{\succeq x'}\subseteq U'$ is an equality. Take any point $u'\in U'$ and let $\ou'$ be its Zariski closure in $X'$. Then the closed set $\ou'\cap T'$ is non-empty since $X'$ is local, and let $y'$ be its maximal point. Pick a valuation scheme $S$ with a morphism $h'\:S\to T'$ taking the generic point to $u'$ and the closed point to $y'$. Set $h:=f\circ h'$, and note that $h^{-1}(T)=h'^{-1}(T')$ is the closed point $s\in S$. Since $X$ is the Ferrand pushout of $U$ and $T$ along $x$, \cite[Lemma~3.3.4]{push} applies to $h$ and we obtain that $\{s\}=h^{-1}(x)$. Thus, $h'(s)\in f^{-1}(x)=\{x'\}$ and we conclude that $u'\succeq h'(s)=x'=y'$.

Let $m\subset A$ and $m_1=mB$ be the ideals corresponding to $x$, and let $m'\subset A'$ and $m'_1=m'B'$ be the ideals corresponding to $x'$. Then $m\otimes_AA'\toisom mA'=m'$ and  $m_1\otimes_BB'\toisom m_1B'=m'_1$ since $f$ is essentially \'etale. By our assumptions, $m=m_1$ in case (i), and $m'=m'_1$ and $f$ is faithfully flat in case (ii). So, in both cases $m=m_1$ and $m'=m'_1$ because the chain $m\subseteq m_1\subset B$ of $A$-modules is taken to $m'\subseteq m'_1\subset B'$ by tensoring with $A'$ over $A$. It also follows that the local ring $R'=A'/m'$ is essentially \'etale over the local ring $R=A/m$.

Since $A$ is the preimage of $R$ in $B$ and $A'$ is the preimage of $R'$ in $B'$, it remains to prove that if $R$ is a valuation ring then so is $R'$, and the converse is true if $R\to R'$ is local. In case (i), this follows from Lemma~\ref{lem:etcovofvalsch}. In case (ii), $R$ is integrally closed and $K=\Frac(R)$ is a field since the same is true for $R'$ and $K'$, and $R\to R'$ is essentially \'etale and local. Thus, $R=R'\cap K$ is a valuation ring.
\end{proof}


\subsubsection{Local case}
We will need the following simple lemma.

\begin{lem}\label{lem:localprincipal}
Let $A$ be a local ring, and  $\{I_\alpha\}$ be a family of ideals in $A$ such that $\sum I_\alpha$ is principal. Then $\{I_\alpha\}$ has a greatest element $I_{\alpha_0}$, in particular $I_{\alpha_0}=\sum I_\alpha$.
\end{lem}
\begin{proof}
Pick a generator $f\in\sum I_\alpha$ and finitely many elements $f_i\in I_{\alp_i}$ such that $f=\sum_{i=1}^k f_i$. Since $I_{\alp_i}\subseteq\sum I_\alpha=Af$, there exist $g_i\in A$ such that $f_i=g_if$, and hence $f(1-\sum g_i)=0$. If $f=0$ then the lemma is clear, otherwise $1-\sum g_i$ is not invertible. Since $A$ is local, at least one $g_i$ does not belong to the maximal ideal of $A$, and hence is invertible. Thus, $f\in I_{\alpha_i}$, which implies the lemma.
\end{proof}

\begin{prop}\label{localpairs}
Theorem~\ref{prufpairth} holds if $X=\Spec(A)$ is a local scheme.
\end{prop}
\begin{proof}
By \S\ref{indepsec}, we shall prove the implications: $(2)\Ra(3)$ and $(8)\Ra(1)$.

$(2)\Ra(3)\:$ The only non-trivial assertion here is that a $U$-flat $U$-admissible morphism $f\:Y\to X$ is flat. The question is flat-local on $Y$ so we may assume that $Y$ is affine, say $Y=\Spec(C)$. The assumption of (2) applied to the identity map $X\to X$ implies that $U=\Spec(B)$ is a local scheme with closed point $u$. Furthermore, $(B,A)$ is a semivaluation ring, so the maximal ideal $m\subset B$ is contained in $A$, $B=A_m$, $R=A/m$ is a valuation ring of $k(u)=B/m$, and $A=B\times_{k(u)}R$. Thus, $A,B,R$ and $k(u)$ form a Ferrand diagram of rings with conductor $m$ in the sense of \cite[\S3.2]{push}.

Set $C_B:=C\otimes_AB$, $C_R:=C\otimes_AR=C/mC$, and $C_u:=C\otimes_Ak(u)$. By \cite[Theorem~2.2(iii)]{Fer}, the homomorphism $\phi\:C\to C_B\times_{C_u}C_R$ is surjective. Since $Y\to X$ is $U$-admissible, the localization homomorphism $\psi\:C\to C_B$ is injective and hence $\phi$ is an isomorphism. By \cite[Theorem~2.2(iv)]{Fer}, it now suffices to prove that $C_B$ is $B$-flat and $C_R$ is $R$-flat. The first holds since $Y$ is $U$-flat. Since $R$ is a valuation ring, the second will follow once we show that $C_R$ is $R$-torsion free. Assume, to the contrary that there exist $\alpha\in A\setminus m$ and $\beta\in C\setminus mC$ such that $\alpha\beta\in mC$. Since $m$ is divisible by any element of $A\setminus m$ (see \S\ref{semisec}), it is divisible by $\alp$ and hence $mC$ is divisible by $\alp$. So, there exists $x\in mC$ such that $\alp x=\alp\beta$. But then $x-\beta$ is a non-zero element annihilated by $\alp$, which contradicts the injectivity of $\psi$.

$(8)\Ra(1)\:$ First, we claim that $U$ is local. Assume to the contrary that $u,u'\in U$ are two distinct closed points, and consider their closures $Z,Z'\subset X$. Clearly, $Z\cap Z'\cap U=\emptyset$, hence also $Z\cap Z'\cap V=\emptyset$ for some quasi-compact open neighborhood $V$ of $U$. By Lemma~\ref{lem:closedsubschfinitetype} below, there exist finite type ideals $\calI,\calI'\subset\calO_V$ with disjoint supports such that $Z\cap V\subseteq V(\calI)$ and $Z'\cap V\subseteq V(\calI')$. By \cite[Theorem~6.9.7]{egaI}, these ideals can be extended to finite type ideals $\calJ,\calJ'\subset\calO_X$, and then the ideal $\calK=\calJ+\calJ'$ is $U$-trivial, hence principal by the assumption of (8). By Lemma~\ref{lem:localprincipal}, either $\calJ=\calK$ or $\calJ'=\calK$, which is a contradiction, since $V(\calI)$ and $V(\calI')$ are disjoint. Thus, the quasi-compact generizing subset $U$ contains at most one closed point. Hence $U=U_{\succeq u}=\Spec(B)$, where $B=\calO_{X,u}$. In particular, $B=A_p$, where $p=m\cap A$ and $m$ is the maximal ideal of $B$.

It remains to prove that $(B,A)$ is a semivaluation ring. First, we claim that $p=m$. Indeed, any $g\in m$ is of the form $s^{-1}f$ with $s\in A\setminus p$ and $f\in p$. The ideal $As+p$ is $U$-trivial, hence principal. Therefore, $p\subseteq As$ by Lemma~\ref{lem:localprincipal}, since $s\notin p$. This shows that $g=s^{-1}f\in m\cap A=p$ as claimed. In particular, $A=\pi^{-1}(R)$, where $\pi\colon B\to B/m$ denotes the natural projection and $R:=\pi(A)\subseteq B/m$. Now, it remains to check that $R$ is a valuation ring.  For any choice of $f,g\in A\setminus p$ the ideal $J:=Af+Ag$ is $U$-trivial, hence principal. Thus, either $f\in Ag$ or $g\in Af$ by Lemma~\ref{lem:localprincipal}, and it follows that for any pair of non-zero elements  $a,b\in R$ either $a\in Rb$ or $b\in Ra$. So, $R$ is a valuation ring and we are done.
\end{proof}


\begin{lem}\label{lem:closedsubschfinitetype}
Let $X$ be a qcqs scheme, and $Z,Z'\subset X$ be closed subschemes such that $Z\cap Z'=\emptyset$. Then there exist finite type ideals $\calI,\calI'\subseteq \calO_X$ such that $Z\subseteq V(\calI)$, $Z'\subseteq V(\calI')$, and $V(\calI)\cap V(\calI')=\emptyset$.
\end{lem}
\begin{proof}
Let $\calJ,\calJ'\subseteq \calO_X$ be the ideal sheaves of $Z$ and $Z'$. By \cite[Corollary~6.9.9]{egaI}, both $\calJ$ and $\calJ'$ are {\em filtered} unions of finitely generated subideals $\calJ_\alp$ and $\calJ'_\beta$, respectively. Thus, $\cap_\alp V(\calJ_\alp)$ and $\cap_\beta V(\calJ'_\beta)$ are disjoint, and, by the quasi-compactness of $X$, there exist $\alp_0,\beta_0$ such that $V(\calJ_{\alp_0})\cap V(\calJ'_{\beta_0})=\emptyset$.
\end{proof}

\subsubsection{Scheme case}\label{schemecasesec}
\begin{prop}\label{schemeprufprop}
Theorem~\ref{prufpairth} holds whenever $X$ is a scheme.
\end{prop}
\begin{proof}
By \S\ref{indepsec} we shall establish two implications: $(2)\Ra(3)$ and $(8)\Ra(1)$.

$(2)\Ra(3)\colon$ We will give an argument that works for an arbitrary algebraic space $X$. Assume that (2) holds, and let $f\:Y\to X$ be a $U$-admissible $U$-flat morphism. Choose an \'etale presentation $h\:X'\to X$, and for any point $x\in X'$ consider the localization $X'_x=\Spec(\calO_{X',x})$. By faithfully flat descent, it suffices to prove that the base change $f_x=f\times_X X'_x$ is flat. Let $U'_x$ be the preimage of $U$ in $X'_x$. Then the pair $(X'_x,U'_x)$ satisfies (1) (for example, because $(\calO_{U',x},\calO_{X',x})$ is a semivaluation ring), and $f_x$ is a $U'_x$-admissible $U'_x$-flat morphism. But we already proved that $(1)\Ra(2)\Ra(3)$ for local schemes, and so $f_x$ is flat.

$(8)\Ra(1)\colon$ Let us show that the pair $(X',f)=(X,id)$ satisfies the requirements of (1). Assume to the contrary that there exists $x\in X$ such that $(X_x,U_x)$ is not obtained from a semivaluation ring, where $X_x:=\Spec(\calO_{X,x})$ and $U_x:=U\cap X_x$. By Proposition~\ref{localpairs}, there exists a $U_x$-trivial finitely generated ideal $I_x\subset\calO_{X,x}$ which is not invertible. To get a contradiction, it remains to extend $I_x$ to a $U$-trivial ideal $I\subset\calO_X$ of finite type.

By the quasi-compactness of $U$, there exists a quasi-compact open neighborhood $W$ of $U$ such that $I_x$ is $W\cap X_x$-trivial. After replacing $U$ by $W$ we may assume that $U$ is open. Let $Z_x$ be the scheme defined by $I_x$, and $Z$ its schematic image in $X$. By Lemma~\ref{lem:extofid} (2) and (3), $Z$ is disjoint from $U$ and its ideal is the colimit of ideals $I_\alp$ of finite type extending $I_x$. By the quasi-compactness of $U$ we can find $I_\alp$ whose support is disjoint from $U$.
\end{proof}

\subsubsection{General algebraic spaces}\label{genalgsec}
In this section we will prove Theorem~\ref{prufpairth} for a general algebraic space $X$. Since the implication $(2)\Ra(3)$ has been proven for such $X$ in the proof of \ref{schemeprufprop}, it remains to establish the implication $(8)\Ra(1)$.

Assume to the contrary that (8) holds but (1) is false for $(X,U)$. Consider an affine presentation $f\: X'\to X$ with a point $x'\in X'$ that violates (1). Set $U':=f^{-1}(U)$. To get a contradiction, we are going to construct an ideal $\calI\subset\calO_X$ that violates (8). The main task is to construct such an ideal on an open subspace, because we can then extend it by \cite[Theorem A]{rydapr} to $X$. If $X_0\subset X$ is open and quasi-compact then $U_0=U\cap X_0$ is also quasi-compact by the quasi-separatedness of $X$. Thus, in the sequel, we may replace $X$ by $X_0$ as above and shrink $U$, $X'$, and $U'$ accordingly whenever the new pair $(X',U')$ violates the conditions of Theorem~\ref{prufpairth}.

Since the identity morphism violates (1) for the scheme $X'$, and the theorem has been proven in this case, there exists a non-invertible finitely generated $U'$-trivial ideal $\calI'\subset\calO_{X'}$. Although $\calI'$ need not be descendable to $X$, it will help us to produce $\calI$ as required. Let $Z'\subset X'$ be the closed locus where $\calI'$ is not invertible. Then $Z'\ne\emptyset=Z'\cap U'$. The image $Z:=f(Z')\subset X$ is constructible by Chevalley's theorem, \cite[$\rm IV_1$, Theorem~1.8.4]{ega}. Our first aim is to shrink $X$ so that $Z$ stays non-empty and becomes closed. Since $Z$ is constructible, it contains all generic points of its closure $\oZ$. By the same reason, $\oZ\setminus Z$ contains all generic points of its closure $\widehat{Z}$, and hence $\widehat{Z}$ does not contain any generic point of $\oZ$. In particular, for the open subspace $V:=X\setminus \widehat{Z}$ we have that $V\cap Z\neq\emptyset$ and $V\cap Z=V\cap\oZ$ is closed in $V$.
Replacing $X$ by an open quasi-compact subspace of $V$ that meets $Z$ we accomplish our goal.

In the sequel, we say that an \'etale morphism $Y\to X$ is an {\em isomorphism} over a subset $T\subseteq |X|$ if for any point $x\in T$ the base change $x\times_XY\to x$ is an isomorphism. In particular, this implies that for any subscheme $Z\into X$ with $|Z|\subseteq T$ the base change $Z\times_XY\to Z$ is an isomorphism. Pick a generic point $\eta\in Z$ and an \'etale morphism $g\: X''\to X$ such that $X''$ is a scheme and $g$ is an isomorphism over $\eta$ (see \S\ref{zarpointssec}). Plainly, $g$ is an isomorphism  over a neighborhood of $\eta$ in $Z$. Thus, after shrinking $X$ again, we may assume that $g$ is surjective and an isomorphism over $Z$. Finally, since we may increase $U$ as long as it stays quasi-compact and disjoint from $Z$, we may assume that $g$ is an isomorphism over $X\setminus U$ thanks to the following:

\begin{claim}
There exists a generizing quasi-compact subset $V$ containing $U$ and disjoint from $Z$ such that $g$ is an isomorphism over $X\setminus V$.
\end{claim}
We postpone the proof of the claim, and first finish the proof of the Theorem. Set $U'':=g^{-1}(U)$ and $\widetilde{X}:=X'\times_XX''$, and let $\widetilde{U}\subset\tilX$ be the preimage of $U$. Since Theorem~\ref{prufpairth} and Corollary~\ref{etpaircor} are already established for schemes, and the pair $(X',U')$ is not Pr\"ufer, it follows that the pair $(\widetilde{X},\widetilde{U})$ is not Pr\"ufer. Therefore, $(X'',U'')$ is not Pr\"ufer either, and there exists a $U''$-trivial, finitely generated, non-invertible ideal $\calI''$ on $X''$. We claim that $\calI''$ is descendable to $X$, which simply follows from the fact that $X''\times_XX''$ is a disjoint union of the diagonal (isomorphic to $X''$) and an open subscheme whose projections onto $X''$ belong to $U''$. So, we obtain an ideal $\calI\subset\calO_X$ with $\calI\calO_{X''}=\calI''$, which violates (8) by \'etale descent. \qed

\begin{proof}[Proof of the Claim.]
Let $X_1$ be the locus of all points $x\in X$ over which the rank of $g$ is at least two. Then $x\in X_1$ if and only if $g$ is not an isomorphism over $x$. Note that $X_1$ is open and quasi-compact because it is the image of the open quasi-compact set $W\subset X''\times_X X''$, which is the union of the components different from the diagonal. Thus, we can choose $V=U\cup X_1$.
\end{proof}

\subsection{Some other properties of Pr\"ufer pairs}
In this section, we describe several properties of Pr\"ufer pairs and spaces, which follow from Theorem~\ref{prufpairth} and Corollary~\ref{prufspaceth}.

\subsubsection{Generizing subsets}
First, let us study generizing subsets of $X$ containing $U$.

\begin{prop}\label{genprop}
Let $(X,U)$ be a Pr\"ufer pair and $V$ be a quasi-compact generizing subset containing $U$. Then $V$ is a pro-open subspace and the pro-open immersion $V\into X$ is affine. In particular, any quasi-compact generizing subset of a Pr\"ufer space is a pro-open subspace with affine immersion morphism.
\end{prop}
\begin{proof}
Since $(X,V)$ is a Pr\"ufer pair, the result follows from the last assertion of Theorem~\ref{prufpairth}.
\end{proof}

\subsubsection{Compatibility with immersions}
Next, we show that Pr\"ufer pairs are compatible with pro-open and closed immersions.

\begin{prop}\label{compproprop}
Assume that $X$ is a qcqs algebraic space and $i\:Y\into X$ is a quasi-compact pro-open immersion. Assume also that $U\subset|X|$ is a generizing subset such that $(X,U)$ is a Pr\"ufer pair. Then $(Y,V)$ is a Pr\"ufer pair, where $V=U\cap|Y|$. In particular, quasi-compact pro-open subspaces of a Pr\"ufer algebraic space are Pr\"ufer.
\end{prop}
\begin{proof}
Choose a presentation $f\:X'\to X$ as in Theorem~\ref{prufpairth}(1) and set $U':=f^{-1}(U)$, $Y':=X'\times_XY$, and $V':=U'\cap|Y'|$. By Theorem~\ref{protop}, $Y'$ is a scheme, hence it suffices to prove that $g\:Y'\to Y$ satisfies condition (1) of Theorem~\ref{prufpairth}. Let $y\in Y'$ be a point. By our assumptions, $y\in Y'\subseteq X'$ possesses a unique minimal generalization $u\in U'$. Since $Y'$ is pro-open in $X'$, it also contains $u$ and hence $u\in V'$. Thus, $y$ possesses a unique minimal generalization $u\in V'$ and the pair $(\calO_{V',u},\calO_{Y',y})=(\calO_{U',u},\calO_{X',y})$ is a semivaluation ring.
\end{proof}

\begin{prop}\label{compclprop}
Let $X$ be a qcqs algebraic space and $i\:Z\into X$ a closed immersion. Assume that $U\subseteq|X|$ and $V\subseteq|Z|$ are schematically dense generizing subsets such that $(X,U)$ is a Pr\"ufer pair and $U\cap|Z|\subseteq V$. Then $(Z,V)$ is a Pr\"ufer pair too. In particular, any integral closed subscheme of a Pr\"ufer algebraic space is Pr\"ufer.
\end{prop}
\begin{proof}
As in the previous proof, we may replace $X$ by an affine presentation. So, $X=\Spec(A)$ and $Z=\Spec(B)$. Assume to the contrary that the pair $(Z,V)$ is not Pr\"ufer. Then there exists a $V$-trivial finitely generated non-invertible ideal $J\subset B$. Lift this ideal to $A$, that is, find a finitely generated ideal $I\subset A$ with $IB=J$. Obviously, $I$ is non-invertible, so to get a contradiction it is enough to show that $I$ could be chosen to be $U$-trivial. Note that $B=A/H$ and $I+H$ is $U$-trivial. It follows from the quasi-compactness of $U$ that already some intermediate finitely generated ideal $I\subseteq I'\subseteq I+H$ is $U$-trivial. Thus, $I'$ is the required lifting of $J$.
\end{proof}

\subsubsection{The retraction $X\to U$}
The following result reduces various questions on the topology of a Pr\"ufer pair $(X,U)$ to the study of the topology of $U$.

\begin{theor}\label{retractth}
Let $(X,U)$ be a Pr\"ufer pair. Then,

(i) Any point $x\in X$ possesses a unique minimal generalization $r(x)\in U$. In particular, $r\:X\to U$ is a retraction of the embedding $U\into X$.

(ii) For any point $u\in U$ the fiber $r^{-1}(u)$ is a Zariski tree. In particular, if $X$ is an integral Pr\"ufer space (resp. a local Pr\"ufer space) then $|X|$ is a Zariski tree (resp. a Zariski chain).

(iii) The retraction $r\:X\to U$ is continuous.
\end{theor}
\begin{proof}
Assume, first, that $X$ is a scheme. In this case, (i) and (ii) follow from Theorem~\ref{prufpairth}(2). To prove (iii) we should show that if $U_0$ is an open subscheme of $U$ then for any point $x\in X$ with $r(x)\in U_0$ there exists a neighborhood $X_0$ of $x$ such that $r(X_0)\subseteq U_0$. Consider the localizations $X_x=\Spec(\calO_{X,x})$ and $U_u=\Spec(\calO_{U,u})\subseteq U_0$. Then $X_x=\cap_i X_i$, where $X_i$ run through all open quasi-compact neighborhoods of $x$. Since $U_u=X_x\cap U$ by Theorem~\ref{prufpairth}(2), we obtain that $U_u$ is the intersection of open quasi-compact subschemes $X_i\cap U$ of $U$, and it follows easily that already some $X_{i_0}\cap U$ lies in $U_0$. Thus, $r(X_{i_0})=X_{i_0}\cap U\subseteq U_0$ and we can take $X_0=X_{i_0}$.

In general, fix a presentation $X'\to X$, and set $U':=f^{-1}(U)$. The pair $(X',U')$ is Pr\"ufer by Corollary~\ref{etpaircor}. For $x\in X$, pick $x'\in f^{-1}(x)$, and let $u'\in U'$ be its minimal generalization that exists by the scheme case. Then $X'_{\succeq x'}\cap X'_{\preceq u'}$ is a Zariski chain. Set $u:=f(u')$. Then, by \'etaleness, $f$ induces an order preserving bijection between $X'_{\succeq x'}\cap X'_{\preceq u'}$ and $X_{\succeq x}\cap X_{\preceq u}$. Hence $X_{\succeq x}\cap X_{\preceq u}$ is a Zariski chain, and $u$ is a minimal generalization of $x$. To see uniqueness, notice that if $u_1\in U$ is a minimal generalization of $x$ then by flatness, there exists $u'_1\in f^{-1}(u_1)\cap X'_{\succeq x'}$.  Thus, $u'_1\succeq u'$ and $u_1\succeq u$. Hence $u_1=u$. This proves assertions (i) and (ii). To prove (iii), notice that the retraction maps commute with $f$, thus the assertion follows from the openness of the map $|X'|\to|X|$.
\end{proof}


\subsubsection{Composing Pr\"ufer spaces}
It is well known that if $C$ is a valuation ring and $B$ is a valuation ring of the residue field $K=C/m$ then the preimage of $B$ in $C$ is a valuation ring $A$ {\em composed} from $C$ and $B$. Similarly, a semivaluation ring is composed from a local ring (its semifraction field) and a valuation ring. These observations can be generalized to Pr\"ufer pairs as follows:

\begin{prop}\label{prufpush}
Assume that $(Z,U)$ is a Pr\"ufer pair, $T=\coprod_{i=1}^nt_i\subset Z$ is a disjoint union of closed Zariski points, and $Y=\coprod_{i=1}^nY_i$, where each $Y_i$ is an integral Pr\"ufer algebraic space with generic point $t_i$. Then $U$ can be identified naturally with a generizing subset of the Ferrand pushout $X=Y\coprod_TZ$ and the pair $(X,U)$ is Pr\"ufer.
\end{prop}
\begin{proof}
In our case Ferrand pushout exists by \cite[Theorem~6.2.1(ii)]{push}. By \cite[Theorem~6.3.5]{push}, $Z\to X$ is a schematically dominant pro-open immersion, hence we can view $U$ as a schematically dense generizing subset of $|X|$. Assume, to the contrary, that $(X,U)$ is not Pr\"ufer and let $f\:X'\to X$ be a non-trivial $U$-modification. Since $(Z,U)$ is Pr\"ufer, $f$ is an isomorphism over $Z$. So, $f$ is an isomorphism over $t_i$'s, and hence it is an isomorphism over $Y_i$'s. In particular, the fibers of $f$ are isomorphisms, and using that $f$ is proper we obtain that it is finite.

To conclude the proof it now suffices to prove the following claim: {\it if the fibers of a schematically dominant finite morphism $f\:X'\to X$ are isomorphisms then $f$ is an isomorphism}. By \'etale descent, we may assume that $X$ is affine. Then $f$ corresponds to a finite injective homomorphism $\phi\:A\to A'$ whose fibers are isomorphisms. Each localization $\phi_p=\phi\otimes_AA_p$ is surjective by Nakayama's lemma, hence each $\phi_p$ is an isomorphism and we obtain that $\phi$ itself is an isomorphism.
\end{proof}

\subsubsection{Disassembling Pr\"ufer spaces}
The following proposition generalizes the well known fact that a valuation ring of finite height $h>1$ is composed from valuation rings of smaller heights.

\begin{prop}\label{disprop} Let $(X,U)$ be a Pr\"ufer pair, and $T=\{t_1\. t_n\}\subset U$ be a finite set of closed points of $U$. Let $Y\subset X$ be the Zariski closure of $T$ provided with the reduced subspace structure. Set $Z:=X\setminus \cup_{i=1}^nX_{\prec t_i}$. If $Z$ is quasi-compact then $Z\to X$ is an affine pro-open immersion and $Y\coprod_TZ\toisom X$ is a Ferrand pushout.
\end{prop}
\begin{proof}

First, $Z\to X$ is an affine pro-open immersion by Proposition~\ref{genprop}. Second, by Propositions~\ref{compproprop} and \ref{compclprop}, $(Z,U)$ and $Y$ are Pr\"ufer. Thus, by Proposition~\ref{prufpush}, the Ferrand pushout $X'=Y\coprod_TZ$ exists, and the pair $(X',U)$ is Pr\"ufer. Furthermore, by \cite[Theorem~6.4.1]{push}, $X'$ is $X$-affine since $Y$ and $Z$ are so. Thus, the morphism $f\:X'\to X$ is separated, quasi-compact, $U$-admissible, and induces an isomorphism over $U$. Hence $f$ is a pro-open immersion by Theorem~\ref{prufpairth}(4). But $|X|=|Y|\cup |Z|$ by the construction. Thus, $f$ is surjective, hence an isomorphism.
\end{proof}

\subsubsection{Separatedness}
It is well known that if $X$ is an integral scheme and $x,y\in X$ are two distinct points such that $\calO_{X,x}=\calO_{X,y}\subset k(X)$ then $X$ is not separated. For Pr\"ufer domains the converse is also true.

\begin{prop}\label{sepprop}
Assume that $f\:X\to S$ is a morphism of schemes whose source is integral and Pr\"ufer. Then $f$ is not separated if and only if there exist points $x,y\in X$ such that $f(x)=f(y)$ and $\calO_{X,x}$ and $\calO_{X,y}$ coincide as subrings of $k(X)$. In particular, $X$ itself is not separated if and only if $\calO_{X,x}$ coincides with $\calO_{X,y}$ for two distinct points of $X$.
\end{prop}
\begin{proof}
Since the assertion is local on $S$, we may assume that $S$ is affine. Hence it is sufficient to consider the absolute case. It is easy to see that in the usual valuative criterion for an integral scheme $Y$, it suffices to test only the valuations of $k(Y)$ (this follows from \cite[\S3.2]{temrz}; see also \cite[Exercise II.4.5(c)]{Hartshorne}). Thus, $X$ is separated if and only if for any valuation ring $\calO$ of $k(X)$ there exists at most one morphism $h\:\Spec(\calO)\to X$ compatible with the generic point. Choosing $h$ is equivalent to choosing a point $x\in X$ such that $\calO$ dominates the local ring $\calO_{X,x}$ in $k(X)$, i.e., the embedding homomorphism $h^*\:\calO_{X,x}\to \calO$ is local. Since $X$ is Pr\"ufer, $\calO_{X,x}$ is a valuation ring, and hence any domination homomorphism $h^*$ as above is an isomorphism. In particular, $\Spec(\calO)$ admits two different morphisms to $X$ if and only if $\calO=\calO_{X,x}=\calO_{X,y}$ for two different points $x,y\in X$.
\end{proof}

\section{Valuation algebraic spaces}\label{valsec}

\subsection{SLP algebraic spaces}

\subsubsection{Basic definitions}
A {\em valuation algebraic space} $X$ is a qcqs Pr\"ufer algebraic space with a unique closed point. By {\em height} of $X$ we mean the topological dimension of $|X|$. In general, a valuation algebraic space does not have to possess an \'etale presentation $U$ which is local (hence a valuation scheme). So, it is more convenient to work with the wider class of {\em semi-local Pr\"ufer algebraic spaces (SLP spaces)}, whose objects are qcqs Pr\"ufer algebraic spaces with finitely many closed points.

\begin{lem}\label{slplem0}
A qcqs algebraic space $X$ is SLP if and only if one (hence any) presentation of $X$ is so.
\end{lem}
\begin{proof}
Follows from Corollary~\ref{etpaircor}.
\end{proof}

\subsubsection{Examples}\label{examsec}
We will prove in Proposition~\ref{slpprop} that any separated SLP space $T$ is an {\em affine} scheme, and so $T$ is the spectrum of a semi-local Pr\"ufer ring. To gain some intuition on non-schematic valuation spaces we consider a few basic examples, which reflect the general behavior of such spaces. We will treat three different examples in a uniform way.

Let $\calO$ be a DVR with field of fractions $K$. Assume that there exist separable quadratic extensions $K_1$, $K_2$, and $K_3$ of $K$ such that the valuation splits in $K_1$, $f_{K_2/K}=2$, and $e_{K_3/K}=2$. For example, such extensions exist for $\calO=\bfZ_{(p)}$. For any $i$, let $\calO_i\subset K_i$ be the integral closure of $\calO$. Set $X:=\Spec(\calO)=\{\eta,s\}$ and $U_i:=\Spec(\calO_i)$. Then $U_1$ has two closed points, and $U_2$ and $U_3$ are local schemes. Set $R_i:=U_i\times_X U_i$. Then $R_i\rra U_i$ are flat equivalence relations with $U_i/R_i\toisom X$; and they are \'etale for $i=1,2$. Plainly, $R_i=R^+_i\cup R^-_i$, where $R_i^+,R_i^-$ are irreducible components and $R_i^+$ denotes the diagonal of the relation $R_i$. Furthermore, the projections $R^+_i\to U_i$ and $R^-_i\to U_i$ are isomorphisms, $R_1^+\cap R_1^-=R_2^+\cap R_2^-=\emptyset$, and $R_3^+\cap R_3^-$ is the closed point of both $R_3^+$ and $R_3^-$.

Let $\eta^-_i$ be the generic point of $R^-_i$ and $R'_i:=R^+_i\coprod\eta^-_i$. Then the natural monomorphism $R'_i\to R_i$ is not an isomorphism and it is a locally closed immersion if $i\le 2$ and bijective for $i=3$. Plainly, the induced projections $R'_i\rra U_i$ are \'etale equivalence relations for any $i$, so $X_i:=U_i/R'_i$ are algebraic spaces. Furthermore, $X_i$ are non-separated SLP spaces, and the natural morphisms $f_i\:X_i\to X$ induced from the morphisms of charts are isomorphisms over $\eta$. The Zariski fibers $f^{-1}_i(s)$ consist of: two points $s_1,s'_1$ with residue field $k(s)$ for $i=1$, one point $s_2$ with larger residue field for $i=2$, and one point $s_3$ with residue field $k(s)$ for $i=3$. Plainly, $X_1$ is a scheme isomorphic to $X$ with doubled closed point, $X_2$ is a valuation algebraic space which is locally separated but not separated, and $X_3$ is a valuation algebraic space which is not locally separated. In particular, $X_i$ are not schemes for $i=2,3$.

\begin{rem}
The first example is classical (but described with an \'etale chart). The second example is essentially Knutson's example of a line with ``twisted" doubled origin, so that non-separatedness of $X_2$ is hidden in ``doubling" the residue field at the origin. The non-separatedness of $X_3$ is in ``replacing" its tangent space by the tangent space of a ramified extension. Note that Knutson considered another basic example of not locally separated algebraic space, namely, the line with ``doubled" tangent space at the origin. The latter is not normal and its ``speciality" can be eliminated by normalization. In particular, such example is irrelevant in the study of valuation algebraic spaces. However, we have seen that not locally separated valuation algebraic spaces exist as well.
\end{rem}

\begin{rem}\label{pushex}
It is easy to see that in the above examples $X_i=U_i\coprod_{\eta_i}\eta$ where $\eta_i=\Spec(K_i)$. On the other hand it is also easy to see that the schematic pushout exists and is equal to $X_i$ for $i=1$ and to $X$ for $i=2,3$. In the latter case, schematic and affine pushouts coincide and are different from the pushout in the category of algebraic spaces.
\end{rem}

\subsubsection{Topology}
Recall that if $X$ is an integral Pr\"ufer space then it is a Zariski tree by Theorem~\ref{retractth}(ii). By quasi-compactness, $X$ is the union of the Zariski chains $X_{\succeq x}$ connecting the generic point of $X$ to its closed points $x$. We aware the reader that even when $X$ is SLP, these chains are not necessarily open.

\begin{lem}\label{simpletoplem}
If $X$ is an SLP space then any generizing subset $U\subset |X|$ is of the form $(\cup_{i=1}^n X_{\succeq x_i})\cup(\cup_{i=n+1}^m X_{\succ x_i})$, where the set $\{x_1\. x_m\}$ is discrete. Furthermore, $U$ is quasi-compact if and only if there exists such a presentation with $m=n$.
\end{lem}
\begin{proof}
We have seen that $X=\cup_{i=1}^n X_{\succeq y_i}$. Each generizing set $U\cap X_{\succeq y_i}$ is of the form $X_{\succeq z_i}$ or $X_{\succ z_i}$ (cf. \S\ref{chainsec}). So, $U$ is of the form $(\cup_{i=1}^n X_{\succeq x_i})\cup(\cup_{i=n+1}^m X_{\succ x_i})$, and we can choose such a representation with minimal possible $m$. Then the set $\{x_i\}$ is discrete, since otherwise some $x_j$ is a specialization of $x_i$, and we can remove the chain going to $x_i$, which contradicts the minimality of $m$. The second claim also follows straightforwardly from \S\ref{chainsec}, and we skip the details.
\end{proof}

\subsubsection{Affine SLP spaces}
Finally, let us give a necessary and sufficient condition for an SLP space to be affine.

\begin{prop}\label{slpprop}
Let $f\:X\to S$ be a morphism of algebraic spaces.

(i) Assume that $X$ is SLP. Then $f$ is separated if and only if $f$ is affine. In particular, any separated SLP space is an affine scheme.

(ii) If $X$ is a valuation space, $f$ is separated, and $S$ is a scheme, then $X$ is affine. In particular, the following conditions on a valuation space are equivalent: $X$ is separated, $X$ is a scheme, $X$ is the spectrum of a valuation ring.
\end{prop}
\begin{proof}
The assertion (i) is \'etale-local on $S$, and the assertion (ii) is Zariski-local on the image of the closed point of $X$ under $f$. In both cases we may assume that $S$ is affine, so it is sufficient to prove the following claim: {\it if $X$ is a separated SLP space then it is an affine scheme}.

Recall that $X$ is a scheme by Proposition~\ref{sepprufprop}. Working separately with connected components reduces the claim to the case of an irreducible $X$. Let $\eta\in X$ be the generic point and let $s_1\.s_n\in X$ be the closed points. Then $\calO_{s_i}$ are valuation rings of the field $k(\eta)$, and using \cite[Ch. VI, \S7.1, Propositions 1 and 2]{Bou} we obtain that $A=\Gamma(\calO_X)=\cap_i\calO_{s_i}$ is a semi-local Pr\"ufer ring whose localizations at the maximal ideals are precisely the rings $\calO_{s_i}$. This implies that the natural morphism $X\to\Spec(A)$ is an isomorphism.
\end{proof}

\subsection{Valuative criteria}\label{valcritsec}

\subsubsection{Valuative diagrams}
Assume that $f\:Y\to X$ is a morphism between algebraic spaces. By an {\em SLP diagram of $f$} we mean an SLP space $T$ with the space of generic points $\eta$, and a pair $\rho$ of compatible morphisms $\rho_Y\:\eta\to Y$ and $\rho_X\:T\to X$. If $T$ is separated then we say that the diagram is {\em classical} since $T$ is a scheme in this case. If only $\rho_X$ is separated then we say that the diagram is {\em separated}, and if $T$ is a valuation space then we say that $\rho$ is a {\em valuative diagram} of $f$. Recall that classical valuative diagrams are used in the usual valuative criteria of properness and separatedness, see \cite[Theorem~7.3, Proposition~7.8]{LMB}. This approach is natural in \'etale topology. On the other hand, we will see that there exist versions of valuative criteria in which $T$ or even $\rho_X$ are allowed to be non-separated but for each point $\eta_i\in\eta$ the morphism $\eta_i\to Y$ is a Zariski point. The SLP diagrams satisfying the latter condition will be called {\em Zariski}.

\subsubsection{Separated pushouts}
It is well known that if $R$ is a valuation ring and $K'$ is a subfield of $K:=\Frac(R)$ then $R':=R\cap K'$ is a valuation ring. It is easy to see that $T'=\Spec(R')$ is the pushout of $T=\Spec(R)$ and $\eta'=\Spec(K')$ under $\eta=\Spec(K)$ in the category of schemes. We will need a generalization of this construction to algebraic spaces, but one has to be careful as even in the above case the pushout in the category of algebraic spaces can be different, see Remark~\ref{pushex}. Because of this we will impose certain separatedness restrictions.

\begin{lem}\label{lem:SLP}
Assume that $X$ is an algebraic space, $f_\eta\:\eta\to\eta'$ is a surjective morphism of algebraic $X$-spaces whose source and target are finite discrete unions of points, and $T$ is an $X$-separated SLP space whose scheme of generic points is $\eta$. Then,

(i) There exists an $X$-separated SLP space $T'$ with scheme of generic points $\eta'$ such that $f_\eta$ extends to a surjective $X$-morphism $f\:T\to T'$.

(ii) The pair $(T',f)$ is unique up to unique $X$-isomorphism.

(iii) $T'$ is the pushout of the diagram $\eta'\la\eta\to T$ in the category of $X$-separated algebraic spaces.
\end{lem}
\begin{proof}
First, let us deduce (iii) from the other assertions. Assume that we are given compatible $X$-morphisms from $\eta$, $\eta'$, and $T$ to an $X$-space $Y$. Then we can view $T$ as a separated $Y$-space, and by the assertion of (i) the morphism $T\to Y$ factors through an SLP space $T''$ such that $T\to T''$ is surjective and $T''\to Y$ is separated. If $Y$ is $X$-separated then $T''$ is also $X$-separated and hence coincides with $T'$ by (ii). In particular, the morphisms to $Y$ we have started with factor through a morphism $T'\to Y$. The latter is unique because $\eta'$ is schematically dominant in $T'$ and $Y$ is $X$-separated.

It remains to prove parts (i) and (ii) of the lemma. Observe that the uniqueness of $T'$ and $f$ asserted in (ii) implies \'etale-functoriality on $X$. Indeed, if $\tilX\to X$ is \'etale, $\tileta$, $\tileta'$, $\tilT$, are the base changes, and $\tilT'$ is the SLP assigned to $\tilX$, $\tilT$, $\tileta$, $\tileta'$ by (i), then $\tilT'=T'\times_X\tilX$. In particular, the assertion of the lemma is \'etale-local on $X$, and we may therefore assume that $X=\Spec(A)$. Furthermore, it is sufficient to consider a single point $\eta'_0$ of $\eta'$ and all connected components of $T$ whose generic points go to $\eta'_0$. So, we may assume that $\eta'$ itself is a point.

Let $\eta_1\.\eta_r$ be the points of $\eta$. By Proposition~\ref{slpprop}, $T$ is an affine SLP scheme. So, $T=\Spec(\prod_{i=1}^r R_i)$, where each $R_i$ is a semi-local Pr\"ufer ring with $\Frac(R_i)=k(\eta_i)$. Recall that $R_i=\cap_{j=1}^{n_i}R_{ij}$, where each $R_{ij}$ is a valuation ring of $k(\eta_i)$ obtained by localizing $R_i$ at a maximal ideal. Each intersection $R'_{ij}:=R_{ij}\cap k(\eta')$ is a valuation ring of $k(\eta')$, hence $R':=\cap_{ij}R'_{ij}$ is a semi-local Pr\"ufer ring with fraction field $k(\eta')$ by \cite[Ch. VI, \S7.1, Proposition1]{Bou}. Note that $R'$ is the intersection of $k(\eta')$ and $\prod_{i=1}^r R_i$ in $k(\eta)=\prod_{i=1}^r k(\eta_i)$. Thus, the homomorphism $A\to\prod_{i=1}^r R_i$ factors through $R'$, and we obtain the required factorization of $f$ through $T'=\Spec(R')$. Obviously, $T'$ is uniquely determined by the two conditions: $T'$ is an affine SLP scheme with generic point $\eta'$, and the morphism $T\to T'$ is surjective.
\end{proof}

\begin{cor}\label{slpcor}
Let $X$ be an algebraic space with Zariski points $x\prec\eta$. Then there exists a valuation space $T$ with generic point $\eta$ and a separated morphism $T\to X$ that extends $\eta\to X$ and takes the closed point of $T$ to $x$.
\end{cor}
\begin{proof}
Pick an \'etale presentation $\tilX\to X$ and lift $x$ and $\eta$ to points $\tilx\prec\tileta$ of $\tilX$. By \cite[Ch. VI, \S1.2, Theorem 1]{Bou}, there exists a valuation ring $R$ of $k(\tileta)$ such that $\tileta\to\tilX$ extends to a morphism $\tilT=\Spec(R)\to\tilX$ taking the closed point to $\tilx$. We obtain a separated morphism $f\:\tilT\to X$, and by Lemma~\ref{lem:SLP}, $f$ factors through a morphism $T\to X$, as required.
\end{proof}

\subsubsection{Zariski-local valuative criterion of separatedness}
Now, we are in a position to prove a criterion of separatedness. Recall that all spaces are assumed to be quasi-separated.

\begin{prop}\label{sepcritprop}
Let $f\:Y\to X$ be a morphism of algebraic spaces, and let $\calC$ be any of the following classes of SLP diagrams $\rho=(\rho_X,\rho_Y)$ of $f$:
\begin{itemize}
\item[(1)] all Zariski valuative diagrams of $f$,
\item[(2)] all SLP diagrams of $f$,
\item[(3)] all classical valuative diagrams of $f$.
\end{itemize}
Then $f$ is separated if and only if for any $\rho\in\calC$, the morphism $\rho_X\:T\to X$ admits at most one lifting $T\to Y$ compatible with $\rho_Y\:\eta\to Y$.
\end{prop}
\begin{proof}
Let us prove the direct implications. It suffices to deal with (2) in this case. Assume that $\rho'$ is an SLP diagram with two distinct liftings $s'_1,s'_2\:T'\to Y$. Pick an affine presentation $g\:T\to T'$, and let $\eta$ be the set of its generic points
\begin{equation}\label{eq1}
\xymatrix{
\eta\ar[d]\ar[r] & \eta'\ar[d]\ar[r] & Y\ar[d] \\
T\ar[r]\ar@{.>}[urr] & T'\ar[r]\ar@{.>}[ur] & X
}
\end{equation}
Since $T\to T'$ is a flat covering, the maps $s_i=s'_i\circ g$ are distinct liftings $T\to Y$. Obviously, they remain distinct after restricting them onto the localization of $T$ at one of its closed points, but this contradicts the usual valuative criterion \cite[Theorem~7.3, Proposition~7.8]{LMB}.

Now, let us prove the opposite implication; in this case (3) is the classical valuative criterion \cite[Theorem~7.3, Proposition~7.8]{LMB} and it suffices to prove (1) since (2) is weaker. Assume, to the contrary, that $f$ is not separated. By the classical criterion, there exists a classical valuative diagram $\rho$ such that there exist two distinct liftings $s_1,s_2\:T\to Y$, and $k(\eta)/k(\eta')$ is a finite separable extension, where $\eta'\to Y$ denotes the Zariski point $\rho_Y$ factors through. So, we get a pushout datum $\eta'\la\eta\into T$ of separated $Y\times_XY$-schemes and by Lemma~\ref{lem:SLP} there exists a pushout $T'=T\coprod_\eta\eta'$ in the category of $Y\times_XY$-separated algebraic spaces. The projections $p_{1,2}\:T'\to Y$ are different since their pullbacks to $T$ are different, hence the morphisms $\eta'\into Y$ and $T'\to X$ form a Zariski valuative diagram of $f$ that admits two different lifts $T'\to Y$.
\end{proof}

\begin{rem}
Although the morphism $T\to Y\times_XY$ in the above proof was separated by the construction, the morphism $T\to X$ could be non-separated. In fact, it is not enough to consider only separated Zariski valuative diagrams in Proposition~\ref{sepcritprop}(1). For example, this is the case when $Y$ and $X$ are valuation spaces and $f$ is a surjective non-separated morphism inducing an isomorphism of generic points, e.g., take $Y=X_2$ or $Y=X_3$ in \S\ref{examsec}. On the positive side, we will prove below that for separated morphisms one can test properness by using only separated diagrams.
\end{rem}

\subsubsection{Zariski-local valuative criterion of universal closedness}\label{zarlocalsec}
By a {\em semivaluation} of $f\:Y\to X$ we mean any separated Zariski valuative diagram of $f$.

\begin{prop}\label{valcritprop}
Let $f\:Y\to X$ be a separated quasi-compact morphism of algebraic spaces and let $\calC$ be any of the following classes of SLP diagrams $\rho=(\rho_X,\rho_Y)$ of $f$:
\begin{itemize}
\item[(1)] all semivaluations of $f$,
\item[(2)] all SLP diagrams of $f$,
\item[(3)] all classical valuative diagrams of $f$.
\end{itemize}
Then $f$ is universally closed if and only if for any $\rho\in\calC$, the morphism $\rho_X\:T\to X$ admits a lifting $T\to Y$ compatible with $\rho_Y\:\eta\to Y$.
\end{prop}
\begin{proof}
We use the same scheme as in the proof of Proposition~\ref{sepcritprop}; in particular, the reader can use diagram (\ref{eq1}) for illustration. It suffices to deal with (2) for the direct implications. Assume that $\rho'$ is an SLP diagram that does not admit liftings $T'\to Y$. Pick an affine presentation $T\to T'$, and let $\eta$ be the set of its generic points. We obtain a classical SLP diagram $\rho_Y\:T\to X$, $\rho_X\:\eta\to Y$ of $f$ and by Lemma~\ref{lem:SLP}(iii), the sets of liftings $T\to Y$ and $T'\to Y$ are naturally bijective. So, $T\to X$ has no liftings to $Y$, and hence the same is true already for a localization of $T$ at one of its closed points. The latter contradicts the classical valuative criterion \cite[Theorem 7.3 and Proposition 7.8]{LMB}.

To prove the inverse implication we note first that (3) is the classical criterion, and (1) implies (2). Assume that $f$ is not universally closed. By the classical criterion, there exists a classical valuative diagram $\rho$ that does not admit a lifting $T\to Y$. By Lemma~\ref{lem:SLP}, $\rho$ factors through a separated Zariski SLP diagram $\rho'$ such that $T\to T'$ is surjective. Then $T'$ is a valuation space, hence $\rho'$ is a semivaluation that does not admit a lifting $T'\to Y$.
\end{proof}

\begin{rem}
(i) If $X$ is separated in Proposition~\ref{valcritprop} then $T$ is separated, hence a valuation scheme. So, in this case we do not have to extend the usual class of valuation spaces to achieve Zariski-locality.

(ii) On the other hand, it may happen for a separated $f$ that the family of its classical semivaluations is not large enough to test universal closedness or properness. For example, this is the case when $X$ is the non-separated valuation space from \S\ref{examsec}, and $Y$ is its generic point

(iii) Although  we have to consider non-separated valuation spaces, the condition that $T\to X$ is separated ensures that they are ``minimally" non-separated.
\end{rem}

\subsubsection{Adic semivaluations}
A separated SLP diagram $\rho$ (see \S\ref{zarlocalsec}) is called {\em adic} if its {\em diagonal} $\eta\to T\times_XY$ is a closed immersion. There is a natural way to assign to any separated valuative diagram of $f$ an adic one.

\begin{lem}\label{valth}
Let $T\to X$, $\eta\to Y$ be a separated valuative diagram of a separated quasi-compact morphism $Y\to X$. Then
\begin{enumerate}
\item[(i)]
There exists a maximal pro-open subspace $U\subseteq T$ such that the $X$-morphism $\eta\to Y$ extends to an $X$-morphism $i\:U\to Y$. Furthermore, the extension $i$ is unique and the space $U$ is quasi-compact.
\item[(ii)]
Let $\eta'$ be the closed point of $U$ and $T'\subseteq T$ be the closure of $\eta'$ equipped with the reduced subspace structure. Then $T'\to X$, $\eta'\to Y$ is an adic valuative diagram.
\end{enumerate}
\end{lem}
\begin{proof}
(i) Uniqueness of $i$ follows from the separatedness of the morphism $Y\to X$. Let $U$ be the schematic image of the diagonal $\eta\to T\times_XY$. Then $U\to T$ is a pro-open immersion by  condition (4) of Corollary~\ref{prufspaceth}, and the projection $U\to Y$ cannot be extended to any larger pro-open subspace of $T$.

(ii) The diagonal $\eta'\to T'\times_XY$ is the base change of the closed immersion $U\to T\times_XY$ with respect to $T'\to T$. So, $(\eta', T')\to (Y,X)$ is an adic valuative diagram.
\end{proof}

\subsubsection{Adic valuative criterion of properness}
We are now in a position to improve the valuative criterion of properness to a criterion that uses only adic semivaluations.

\begin{theor}\label{valcritth}
Let $f\:Y\to X$ be a separated morphism of finite type of qcqs algebraic spaces. Then $f$ is proper if and only if for any adic semivaluation $\rho$ of $f$ the morphism $T\to X$ admits a unique lifting $T\to Y$ compatible with $\eta\to Y$.
\end{theor}
\begin{proof}
In view of Proposition~\ref{valcritprop} we shall only prove that if any adic semivaluation admits a lifting then the same is true for an arbitrary semivaluation $T\to X$, $\eta\to Y$. Let $\rho'_X\:T'\to X$ and $\rho'_Y\:\eta'\to Y$ be the adic valuative diagram associated to $\rho$ by Lemma~\ref{valth}. By Lemma~\ref{lem:SLP}, $\rho'$ factors through a semivaluation $\rho''_X\:T''\to X$, $\rho''_Y\:\eta''\to Y$, and it is easy to see that $\rho''$ is also adic.
$$
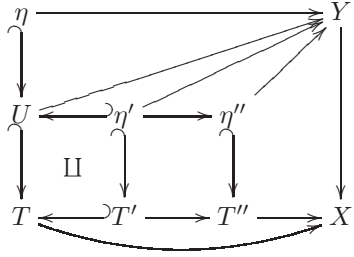
\xymatrix{
\eta \ar@{_{(}->}[d]\ar[rrr] & & & Y\ar[dd]\\
U\ar@{_{(}->}[d]\ar[urrr]\ar@{}[dr] |\coprod & \eta'\ar@{_{(}->}[d]\ar[r]\ar@{_{(}->}[l]\ar[urr] & \eta''\ar@{_{(}->}[d]\ar[ur] & \\
T \ar@/_1pc/[rrr] & T' \ar@{_{(}->}[l]\ar[r] & T''\ar[r] & X
}
\vspace{0.4cm}
$$
A lifting $T''\to Y$ exists by our assumptions, hence we obtain a lifting $T'\to T''\to Y$. The latter lifting and the morphism $i\:U\to Y$ induce a morphism $T\to Y$ because $T'\coprod_{\eta'} U\toisom T$ by Proposition~\ref{disprop}. Obviously, $T\to Y$ is an $X$-morphism that extends $\eta\to Y$, hence it is the required lifting, and we are done.
\end{proof}
\begin{rem}
By using Proposition~\ref{disprop} we invoke the results of \cite{push}, but this is done for the sake of convenience only. In fact, one can show that $T'\coprod_{\eta'} U\toisom T$ by a simple computation with presentations.
\end{rem}
\subsubsection{Valuative obstacle for the existence of schematization}
We proved in Proposition~\ref{sepprufprop} that any separated integral qcqs algebraic space $X$ can be modified to a scheme by a blow up $X'\to X$. In general, non-separated valuation spaces provide an obstacle for schematization.

\begin{lem}\label{obstlem}
Let $X$ be a qcqs algebraic space with an open subspace $U$. If there exists a semivaluation $\rho_U\:\eta\to U$, $\rho_X\:T\to X$ with a non-separated valuation space $T$ then for any $U$-admissible blow up $X'\to X$ the space $X'$ is not a scheme.
\end{lem}
\begin{proof}
By Proposition~\ref{valcritprop}(2), there exists a lifting $T\to X'$. Thus, $X'$ is not a scheme, since otherwise $T$ would be separated by Proposition~\ref{slpprop}(ii).
\end{proof}

\begin{rem}
(i) With some extra-work one can obtain an analogous relative result on schematization of a morphism $X\to S$ whose restriction $U\to S$ onto an open subspace is schematic. We prefer not to pursue this direction for the sake of simplicity.

(ii) We will prove in \cite{nag} that the converse of Lemma~\ref{obstlem} holds true, thus obtaining a criterion for existence of schematization. Roughly speaking, this means that non-separated valuation spaces are the only obstacles for schematization.

(iii) An analogous obstruction based on non-separated valuation spaces was (implicitly) used in \cite{ct}. The main results of \cite{ct} are that if $k$ is a non-archimedean field then: (i) any \'etale equivalence relation $R\rra U$ on Berkovich $k$-analytic spaces is effective whenever its diagonal is closed (and so $U/R$ is a separated analytic space), (ii) any separated algebraic space over $k$ is analytifiable by a $k$-analytic space (or a rigid space). In addition, it was shown by various examples that the separatedness assumptions (including closedness of the diagonal) cannot be ignored. The central mechanism of all those examples was a non-separated valuation space that obstructed analytification pretty similarly to the obstruction to schematization in Lemma~\ref{obstlem}. Since it was implicit in \cite{ct}, let us comment this briefly. Example \cite[Example~3.1.1]{ct} used valuations that correspond to analytic but non-rigid points, so it could be naturally interpreted in the analytic (but not rigid) geometry. Examples \cite[Examples~5.1.3, 5.1.4]{ct} were based on height two valuations that correspond to non-analytic adic Huber's point, so they could only be interpreted using the germ reduction functor or Huber's adic spaces.

(iv) Let $R\rra U$ be as in (iii). It seems probable that the existence of non-separated semivaluations is the only obstacle for the quotient $U/R$ to be effective, but we did not try to check this. If true, this will also imply that if an algebraic $k$-space $\calX$ possesses a schematization then $\calX$ also possesses an analytification. By Chow's lemma any separated $\calX$ possesses a schematization, hence this would improve the analytification criterion of \cite{ct}, though would not provide a complete criterion. Indeed, if $\calX$ is obtained from an algebraic surface $X$ by multiplying a curve $C\into X$ via a finite \'etale covering $f\:\calC\to C$ then it is easy to see that $\calX$ admits a schematization if and only if $f$ is a local isomorphism, i.e. $\calC$ is a disjoint union of copies of $C$. On the other hand, the argument of \cite[Example~3.1.1]{ct} shows that $\calX$ admits an analytification if and only if the analytic covering $f^\an\:\calC^\an\to C^\an$ is a local isomorphism, which can freely happen even when $f$ is non-trivial.
\end{rem}

\subsubsection{Uniformizable SLP spaces}\label{unifsec}
The following class of SLP's will play an important role in \cite{nag}. We say that an SLP space $X$ is {\em uniformizable} if the set $\eta$ of its generic points is open. This happens if and only if $|X|\setminus\eta$ has no unbounded generalizing sequences $x_1\prec x_2\prec\dots$. Equivalently, any non-generic point of $|X|$ is a specialization of a codimension one point. Note that any SLP space has finitely many such points. A typical example of a uniformizable SLP space is an SLP space of finite height. Plainly, an SLP space is uniformizable if and only if one, and hence any, of its \'etale coverings is so. An affine SLP space $X=\Spec(A)$ is uniformizable if and only if there exists an element $\pi\in A$ such that the localization $A_\pi$ coincides with the total ring of fractions of $A$. In this case, $A_\omega=\Frac(A)$ if and only if $\omega$ is contained in all prime ideals of height one. If, moreover, $A$ is a valuation ring then an element $\pi$ as above is called a {\em uniformizer} (although, this contradicts the classical terminology when $A$ is a DVR). We will use uniformizable SLP spaces and the following uniformizability criterion in \cite{nag}.

\begin{lem}\label{uniflemma}
Let $f\:Y\to X$ be a finite type morphism of algebraic spaces, and  $T\to X$, $\eta\to Y$ be an adic Zariski SLP diagram of $f$. Then $T$ is uniformizable.
\end{lem}
\begin{proof}
The schematically dominant pro-open immersion $\eta\to T$ is of finite type, since  $\eta\to Y\times_XT\to T$ are so. Hence it is open by Corollary~\ref{factorcor}(ii).
\end{proof}

\begin{rem}
(i) What we call here a {\em uniformizable valuation ring} $A$, i.e., $\Spec(A)$ is uniformizable, is sometimes called a microbial valuation ring, and a uniformizer $\pi\in A$ is sometimes called a microbe.

(ii) Note that $\pi$ is a uniformizer if and only if $\cap_{n\in\bfN}(\pi^n)=0$. The $(\pi)$-adic completion $\hatA$ does not depend on the choice of a uniformizer $\pi$ and it is the only adic completion whose separated completion homomorphism $A\to\hatA$ is injective. So, the only valuation rings that possess a reasonable completion theory are the uniformizable valuation rings.
\end{rem}

\appendix
\section{Approximation for pro-open immersions}
The aim of the appendix is to show that if a pro-open subspace $U$ is the intersection of open subspaces $U_\alp$ then the usual approximation theory applies to the limit $U=\lim_\alp U_\alp$. This task is not so simple because the transition morphisms do not have to be affine and it is even unclear whether there exists a cofinal subfamily with affine transition morphisms.

\subsection{Preparation}
Note that in the following result it is not enough to assume that $f$ is of finite type as can be easily seen using examples from \S\ref{examsect}.

\begin{prop}\label{factorprop}
Assume that $U,Y,X$ are qcqs algebraic spaces, $j\:U\to Y$ is a schematically dominant morphism, $f\:Y\to X$ is a separated morphism such that the composition $i\:U\to X$ is a pro-open immersion, and one of the following conditions is satisfied:
\begin{itemize}
\item[(i)] $f$ is of finite presentation,

\item[(ii)] $f$ is schematically dominant and of finite type.
\end{itemize}
Then, for a sufficiently small quasi-compact open neighborhood $V\subset X$ of $i(U)$, the restriction $f\times_XV$ is an isomorphism.
\end{prop}
\begin{proof}
In view of Proposition~\ref{proopenlocus} it suffices to show that the pro-open locus of $f$ contains $i(U)$. After replacing $X$ by a presentation we may assume that $X$ is a scheme. Choose a Zariski point $u\to U$ and set $y=j(u)$, $x=i(u)$, $X_x=\Spec(\calO_{X,x})$, $Y_x=Y\times_XX_x$ and $U_x=U\times_XX_x$. We should prove that the morphism $f_x\:Y_x\to X_x$ is an isomorphism. Note that $U_x\to X_x$ is a surjective pro-open immersion by claims (i) and (iv) of Proposition~\ref{proprop}, hence an isomorphism. Hence $f_x$ is an isomorphism by the following lemma.
\end{proof}

\begin{lem}\label{factorlem}
Let $\pi\:Y\to X$ be a separated morphism of qcqs algebraic spaces, and $s\:X\to Y$ a section of $\pi$. Then $s$ is a closed immersion. In particular, if $s$ is schematically dominant then $\pi$ and $s$ are isomorphisms.
\end{lem}
\begin{proof}
Note that $s$ is the equalizer of the morphisms $s\circ\pi$ and $\id_Y$. Indeed, any morphism $f\:Z\to Y$ with $f=s\circ\pi\circ f$ factors as $Z\stackrel{\pi\circ f}\to X\stackrel{s}\to Y$, and this gives rise to the identification $X={\rm Eq}(s\circ\pi,\id_Y)$. On the other hand, the equalizer can be expressed as the base change of the morphism $(s\circ\pi,\id_Y)\:Y\to Y\times_XY$ with respect to the diagonal $\Delta\:Y\to Y\times_XY$. Since $\pi$ is separated, $\Delta$ is a closed immersion, and hence $s$ is so.
\end{proof}

\subsection{Pro-open approximation}
Now, we can extend the approximation theory to pro-open immersions.

\begin{theor}\label{factorcor1}
Assume that $i\:U\to X$ is a schematically dominant pro-open immersion of qcqs algebraic spaces, and $\{U_\alp\}_{\alp\in A}$ is the family of open qcqs subspaces of $X$ with $|U|\subset|U_\alp|$. Then,

(i) $U_\alp$ are cofinal among the family $\{Y_\beta\}_{\beta\in B}$ of all strict $U$-quasi-modifications of $X$.

(ii) All results of the classical approximation theory mentioned in \S\ref{approxsec} hold true for the limit $U=\lim_{\alp\in A}U_\alp$.
\end{theor}
\begin{proof}
Assertion (i) follows from the definition of quasi-modifications, see \S\ref{qmodsec}. Let us use it to prove assertion (ii). By Lemma~\ref{thD}, $U$ is isomorphic to the limit of a family $\{Z_\gamma\}_{\gamma\in C}$ of $X$-separated $X$-spaces of finite type with affine and schematically dominant transition morphisms. In particular, the projections $U\to Z_\gamma$ are schematically dominant, so each $Z_\gamma$ is a strict $U$-quasi-modification of $X$ by Proposition~\ref{factorprop}(ii). By Lemma~\ref{propB}, any morphism $U\to Y_\beta$ factors through some $Z_\gamma$, so the family $\{Z_\gamma\}$ is cofinal in the family $\{Y_\beta\}$. The transition morphisms in the former family are affine, so the classical approximation applies to it, and by the cofinality of $\{Z_\gamma\}$ and $\{U_\alp\}$ in $\{Y_\beta\}$, all approximation results apply also to the families $\{Y_\beta\}_{\beta\in B}$ and $\{U_\alp\}_{\alp\in A}$.
\end{proof}

\bibliographystyle{amsalpha}
\bibliography{nagata}

\end{document}